\documentclass{conm-p-l}
\usepackage{amssymb}
\usepackage{graphicx}
\usepackage[colorlinks,citecolor=blue,linkcolor=blue,urlcolor=blue,filecolor=blue,breaklinks]{hyperref}
\usepackage[utf8]{inputenc}
\usepackage{amsfonts}
\usepackage{amsmath}
\usepackage{amsthm}
\usepackage{mathtools}
\usepackage{mathrsfs}
\usepackage{relsize}
\usepackage{tikz}
\usetikzlibrary{arrows,calc,cd}

\theoremstyle{plain}

\newtheorem{theorem}{Theorem}[section]
\newtheorem{lemma}[theorem]{Lemma}
\newtheorem{corollary}[theorem]{Corollary}
\newtheorem{proposition}[theorem]{Proposition}

\theoremstyle{definition}
\newtheorem{definition}[theorem]{Definition}

\theoremstyle{remark}
\newtheorem{remark}[theorem]{Remark}

\numberwithin{equation}{section}

\newcommand{\R}{\mathbb{R}}
\newcommand{\Z}{\mathbb{Z}}

\newcommand{\Heis}{\mathcal{H}}
\newcommand{\Heisr}{\mathcal{H}_{\mathbb{R}}}

\newcommand{\cL}{\mathcal{L}}
\newcommand{\longhookrightarrow}{\ensuremath{\lhook\joinrel\longrightarrow}}
\newcommand{\incl}[3][right]%
{%
\draw[<-,>=#1 hook] #2 to ($ #2!0.5!#3 $);
\draw[->] ($ #2!0.5!#3 $) to #3;%
}

\begin{document}

\title[Action of subgroups of the MCG on Heisenberg homologies]{Action of subgroups of the mapping class group on Heisenberg homologies}

%    author one information
\author{Christian Blanchet}
\address{Universit{\'e} Paris Cit{\'e} \& Sorbonne Universit{\'e}, CNRS, IMJ-PRG, F-75006 Paris, France}
\curraddr{}
\email{christian.blanchet@imj-prg.fr}
\thanks{The first and third authors are thankful for the support of the Abdus Salam School of Mathematical Sciences.}

%    author two information
\author{Martin Palmer}
\address{Simion Stoilow Mathematical Institute of the Romanian Academy, Bucharest, Romania}
\curraddr{}
\email{mpanghel@imar.ro}
\thanks{The second author was partially supported by a grant of the Romanian Ministry of Education and Research, CNCS - UEFISCDI, project number PN-III-P4-ID-PCE-2020-2798, within PNCDI III}

%    author three information
\author{Awais Shaukat}
\address{Abdus Salam School of Mathematical Sciences, Lahore, Pakistan}
\curraddr{}
\email{m.awais\_shaukat@live.com}
\thanks{}


\date{16 October 2023}

\begin{abstract}
In previous work we constructed twisted representations of mapping class groups of surfaces, depending on a choice of representation $V$ of the Heisenberg group $\Heis$. For certain $V$ we were able to untwist these mapping class group representations. Here, we study the restrictions of our twisted representations to different subgroups of the mapping class group. In particular, we prove that these representations may be untwisted on the Torelli group for any given representation $V$ of $\Heis$. When $V$ is the Schrödinger representation, we also construct untwisted representations of subgroups defined as kernels of crossed homomorphisms studied by Earle and Morita.
\end{abstract}

\maketitle

\section*{Introduction}

In recent work \cite{HeisenbergHomology}, we constructed a twisted action of the mapping class group of any compact, connected, oriented surface $\Sigma$ with one boundary component on the homology of configuration spaces with local coefficients determined by a representation $V$ of the discrete Heisenberg group $\Heis = \Heis(\Sigma)$. The details of this construction are recalled briefly in \S\ref{s:twisted}. For specific representations $V$ of $\Heis$ we were able to untwist and obtain genuine, untwisted linear representations of the mapping class group (for the linearisation $\Heis \oplus \Z$ of the affine translation action of $\Heis$ on itself) or linear representations of central extensions of the mapping class group (for the Schrödinger representation of $\Heis$).

Our goal here is to complete the study of this action on Heisenberg homology.
In \S\ref{s:Chillingworth} we identify the kernel of the action of the mapping class group on the Heisenberg group as the Chillingworth subgroup (Proposition \ref{kernel_of_Psi}); hence we obtain a linear representation of this subgroup for any representation $V$ of $\Heis$ (Theorem \ref{thm:Chill}).
We also identify the \emph{projective kernel} of this action (i.e.\ the subgroup of elements that act by inner automorphisms) with the Torelli group (Proposition \ref{Psi_inner}) and use this fact in \S\ref{s:Torelli} to obtain untwisted linear representations of the Torelli group for any $V$ (Theorem \ref{thm:Torelli}). In the special case where $V$ is the Schr{\"o}dinger representation -- where in \cite{HeisenbergHomology} we obtained an action of the stably universal central extension of the mapping class group -- we show in \S\ref{s:Morita} that, restricting to a so-called \emph{Earle-Morita subgroup} defined in Definition \ref{def:Earle-Morita}, we obtain linear representations without passing to any extension (Theorem \ref{thm:Morita}).

Using the local system given by the Schr{\"o}dinger representation at an odd root of $1$, De Renzi and Martel \cite{DeRenziMartel2022} produced a homological model for TQFT representations derived from quantum $sl(2)$. Our results shed light on a few points in their paper. First, Proposition \ref{kernel_of_Psi} identifies a certain subgroup of the mapping class group denoted by $\mathscr{M}_g^{\mathbb{H}}$ in \cite[Proposition 2.21]{DeRenziMartel2022} with the Chillingworth subgroup. Second, the construction of \cite[\S 6.2]{DeRenziMartel2022} depends on the identification (Proposition \ref{Psi_inner}) of the projective kernel of the action of the mapping class group on $\Heis$ with the Torelli group.

\section{Twisted representations of the mapping class group}
\label{s:twisted}

\subsection{A review of Heisenberg homology}
\label{subsec:review}

Let $\Sigma=\Sigma_{g,1}$, for $g\geq 1$, be a compact, connected, oriented surface of genus $g$ with one boundary component. For $n\geq 2$, the unordered configuration space of $n$ points in $\Sigma_{g,1}$ is
\[
\mathcal{C}_{n}(\Sigma_{g,1} )= \{ \{c_{1},\dots,c_{n}\} \subset \Sigma_{g,1} \mid c_i\neq c_j \text{ for $i\neq j$}\}.
\]
The surface braid group is then defined as $\mathbb{B}_{n}(\Sigma)=\pi_{1}(\mathcal{C}_{n}(\Sigma),*)$. 
A presentation for this group was first obtained by G.~P.~Scott \cite{Scott} and subsequently revisited by Gonz\'alez-Meneses \cite{Gonzalez} and Bellingeri \cite{Bellingeri}.
We fix a collection of based loops, $\alpha_1,\ldots,\alpha_g,\beta_1,\ldots,\beta_g$, as depicted in Figure \ref{modelSurf}.
The base point $*_1$ belongs to the base configuration $*$. We will use the same notation $\alpha_r$, $\beta_s$ for the corresponding braids where only the first point is moving.

\begin{figure}[h]
\centering
\includegraphics[scale=0.65]{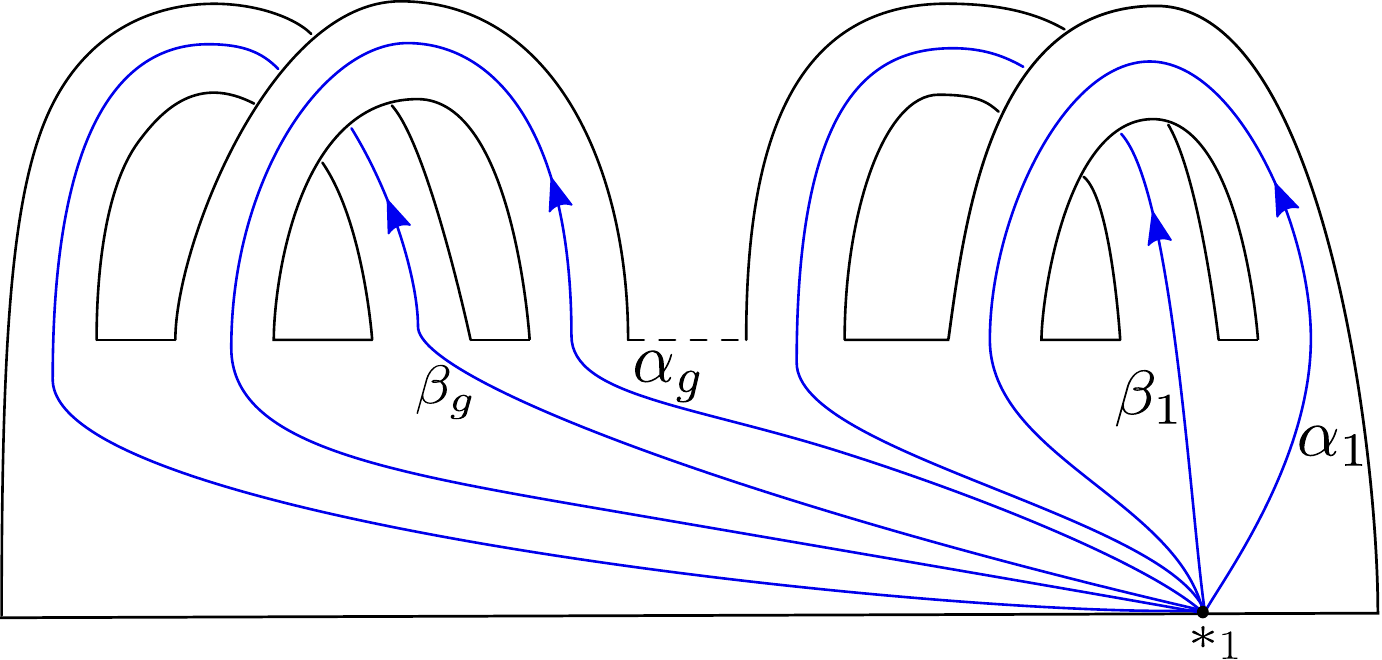}
\caption{Model for $\Sigma$.}
\label{modelSurf}
\end{figure}

The braid group $\mathbb{B}_{n}(\Sigma)$ has generators $\alpha_1,\ldots,\alpha_g$, $\beta_1,\ldots,\beta_g$, together with the classical generators $\sigma_1,\ldots,\sigma_{n-1}$,  and relations:
\begin{equation}
\label{eq:relations}
\begin{cases}
\,\text{(\textbf{BR1}) }\, [\sigma_{i},\sigma_{j}] = 1 & \text{for } \lvert i-j \rvert \geq 2, \\
\,\text{(\textbf{BR2}) }\, \sigma_{i}\sigma_{j}\sigma_{i}=\sigma_{j}\sigma_{i}\sigma_{j} & \text{for } \lvert i-j \rvert = 1, \\
\,\text{(\textbf{CR1}) }\, [\zeta,\sigma_{i}] = 1 & \text{for } i>1 \text{ and all } \zeta\text{ among the }\alpha_r,\beta_s,\\
\,\text{(\textbf{CR2}) }\, [\zeta,\sigma_{1}\zeta\sigma_{1}] = 1 & \text{for all } \zeta\text{ among the }\alpha_r,\beta_s, \\
\,\text{(\textbf{CR3}) }\, [\zeta,\sigma^{-1}_{1}\eta\sigma_{1}] = 1& \text{for all } \zeta\neq \eta \text{ among the }\alpha_r,\beta_s,\text{ with}\\
& \{\zeta,\eta\}\neq \{\alpha_r,\beta_r\},\\
\,\text{(\textbf{SCR}) }\, \sigma_{1}\beta_{r}\sigma_{1}\alpha_{r}\sigma_{1}=\alpha_{r}\sigma_{1}\beta_{r} & \text{for all } r.
\end{cases}
\end{equation}
Composition of loops is written from right to left.

We will use the notation $x.y$ for the standard intersection form on $H_1(\Sigma;\Z)$. The \emph{Heisenberg group} $\Heis(\Sigma)$ is the central extension of the homology group $H_1(\Sigma;\Z)$ defined using the $2$-cocycle $(x,y)\mapsto x.y$. As a set, it is equal to $\Z \times H_1(\Sigma;\Z)$, and the operation is given by
\begin{equation}
\label{eq:Heisenberg-product}
(k,x)(l,y)=(k+l+\,x.y,x+y).
\end{equation}

We will often denote the Heisenberg group simply by $\Heis = \Heis(\Sigma)$ when the surface $\Sigma$ under consideration is clear. We will use the notation $a_r$, $b_s$, for the homology classes of $\alpha_r$, $\beta_s$. From the presentation \eqref{eq:relations} we deduced the following in \cite[\S 1]{HeisenbergHomology}.
\begin{proposition}
\label{hom_phi}
For each $g\geq 1$ and $n\geq  2$, the quotient of the braid group $\mathbb{B}_n(\Sigma)$ by the subgroup  $[\sigma_1,\mathbb{B}_n(\Sigma)]^N$ normally generated by the commutators $[\sigma_1,x]$, $x\in \mathbb{B}_n(\Sigma)$, is isomorphic to the Heisenberg group $\Heis(\Sigma)$. An isomorphism
\begin{equation}
\label{eq:Heisenberg-quotient-isomorphism}
\mathbb{B}_n(\Sigma)/[\sigma_1,\mathbb{B}_n(\Sigma)]^N \cong \Heis(\Sigma)
\end{equation}
is represented by the surjective homomorphism
\[
\phi \colon \mathbb{B}_{n}(\Sigma) \relbar\joinrel\rightarrow \Heis(\Sigma)
\]
sending each $\sigma_i$ to $u=(1,0)$, $\alpha_r$ to $\tilde{a}_r=(0,a_r)$ and $\beta_s$ to $\tilde{b}_s=(0,b_s)$.
\end{proposition}

From the homomorphism $\phi$ we obtain a regular covering $\widetilde{\mathcal{C}}_n(\Sigma)$ of the configuration space $\mathcal{C}_n(\Sigma)$. The homology of this covering space is the homology of $\mathcal{C}_n(\Sigma)$ with local coefficients defined by $\phi$, which we call \emph{Heisenberg homology} and denote by $H_*(\mathcal{C}_n(\Sigma),\Z[\Heis])$. It is equipped with a right $\Z[\Heis]$-module structure defined by deck transformations. 
 
Let us denote by $\mathcal{S}_*(\widetilde{\mathcal{C}}_n(\Sigma))$ the singular chain complex of the Heisenberg covering $\widetilde{\mathcal{C}}_n(\Sigma)$; this is a complex of right $\Z[\Heis]$-modules. Given a (left) representation $\rho\colon \Heis \to GL(V)$, the corresponding twisted homology is that of the complex
\begin{equation}
\label{eq:Local}
\mathcal{S}_*(\mathcal{C}_n(\Sigma);V) := \mathcal{S}_*(\widetilde{\mathcal{C}}_n(\Sigma))\otimes_{\Z[\Heis]} V
\end{equation}
This will be called the \emph{Heisenberg homology} of surface configurations with coefficients in $V$.

We also consider the Borel-Moore homology
\begin{equation}\label{eq:BorelMoore}
H_*^{BM}(\mathcal{C}_{n}(\Sigma);V) = 
{\varprojlim_T}\, H_*(\mathcal{C}_{n}(\Sigma), \mathcal{C}_{n}(\Sigma)\setminus T ; V),
\end{equation}
where the inverse limit is taken over all compact subsets  $T\subset\mathcal{C}_{n}(\Sigma)$. 
We denote by $\mathcal{C}_{n}(\Sigma,\partial^-(\Sigma))$ the closed subspace of configurations containing at least one point in a fixed closed interval $\partial^-(\Sigma)\subset \partial \Sigma$. The relative Borel-Moore homology is defined similarly as
\begin{equation}
\label{eq:relativeBorelMoore}
\begin{multlined}
H_*^{BM}(\mathcal{C}_{n}(\Sigma),\mathcal{C}_{n}(\Sigma,\partial^-(\Sigma));V) \\
\qquad\qquad = {\varprojlim_T}\, H_* \bigl( \mathcal{C}_{n}(\Sigma), \mathcal{C}_{n}(\Sigma,\partial^-(\Sigma)) \cup (\mathcal{C}_{n}(\Sigma)\setminus T) ; V \bigr) .
\end{multlined}
\end{equation}
The following theorem computes the relative Borel-Moore homology as a module.

\begin{theorem}[{\cite[\S 2]{HeisenbergHomology}}]
\label{basis}
Let $V$ be any representation of the discrete Heisenberg group $\Heis(\Sigma)$. Then, for $n\geq 2$, there is an isomorphism of modules
\[
H_n^{BM}(\mathcal{C}_{n}(\Sigma),\mathcal{C}_{n}(\Sigma,\partial^-(\Sigma)) ;V) \;\cong\; \bigoplus_{k\in \mathcal{K}} V.
\]
Furthermore, this is the only non-vanishing module in $H_*^{BM}(\mathcal{C}_{n}(\Sigma),\mathcal{C}_{n}(\Sigma,\partial^-(\Sigma)) ;V)$.
\end{theorem}

There is also an explicit geometric description of the indexing set $\mathcal{K}$ of the direct sum decomposition of Theorem~\ref{basis}, which is explained in \cite[\S 2]{HeisenbergHomology}; however, it will not be essential for the present paper.

\subsection{An action of the mapping class group on the Heisenberg group}

The \emph{mapping class group} of $\Sigma$, denoted by $\mathfrak{M}(\Sigma)$, is the group of orientation-preserving diffeomorphisms of $\Sigma$ fixing the boundary pointwise, modulo isotopies relative to the boundary. The isotopy class of a diffeomorphism $f$ is denoted by $[f]$. An orientation-preserving self-diffeomorphism $f \colon \Sigma \rightarrow \Sigma$ fixing the boundary pointwise gives a homeomorphism $\mathcal{C}_{n}(f) \colon \mathcal{C}_{n}(\Sigma) \rightarrow \mathcal{C}_{n}(\Sigma)$, defined by $\{x_{1},x_{2},\ldots ,x_{n}\} \mapsto \{f(x_{1}),f(x_{2}),\ldots ,f(x_{n})\}$. If we ensure that the basepoint configuration of $\mathcal{C}_n(\Sigma)$ is contained in $\partial\Sigma$, then it is fixed by $\mathcal{C}_n(f)$ and this in turn induces an automorphism $f_{\mathbb{B}_{n}(\Sigma)} = \pi_{1}(\mathcal{C}_{n}(f)) \colon \mathbb{B}_{n}(\Sigma) \rightarrow \mathbb{B}_{n}(\Sigma)$, which depends only on the isotopy class $[f]$ of $f$. In \cite[\S 3]{HeisenbergHomology} we proved:

\begin{proposition}
\label{f_Heisenberg}
There exists a unique automorphism $f_{\Heis} \colon \Heis \rightarrow \Heis$ such that the following square commutes:
\begin{equation}
\label{eq:projection-equivariance}
  \begin{tikzcd}
     \mathbb{B}_{n}(\Sigma) \arrow[d,swap, "\phi"] \arrow[rr, "f_{\mathbb{B}_{n}(\Sigma)}"] && \mathbb{B}_{n}(\Sigma) \arrow[d,"\phi"] \\
     \Heis \arrow[rr, "f_{\Heis}"] && \Heis
  \end{tikzcd}
\end{equation}
Thus, there is an action of $\mathfrak{M}(\Sigma)$ on the Heisenberg group $\Heis$ given by
\begin{equation}\label{eq:action_on_Heis}
\Psi \colon f \mapsto f_\Heis \colon \mathfrak{M}(\Sigma) \longrightarrow \mathrm{Aut}(\Heis).
\end{equation}
\end{proposition}

\subsection{Twisted representations of the mapping class group.}
\label{subsec:twisted-representations}

In \cite{HeisenbergHomology} we explained how to use the action \eqref{eq:action_on_Heis} to obtain twisted representations of $\mathfrak{M}(\Sigma)$ for each representation $V$ of $\Heis$ over a ring $R$ and each integer $n\geq 2$. We recall this briefly here.

For a (left) representation $\rho \colon \Heis \to \mathrm{Aut}_R(V)$ and an automorphism $\tau \in \mathrm{Aut}(\Heis)$, the $\tau$-twisted representation $\rho \circ \tau$ is denoted by ${}_\tau \! V$. Also, for any representation $V$ of $\Heis$, we denote the induced local system
\[
\Z[\widetilde{\mathcal{C}}_n(\Sigma)] \otimes_{\Z[\Heis]} V
\]
on the configuration space $\mathcal{C}_n(\Sigma)$ simply by $V$, by abuse of notation. We then write
\begin{equation}
\label{eq:Borel-Moore-homology-group}
\mathcal{V}_n(V) = H_n^{BM} (\mathcal{C}_n(\Sigma) , \mathcal{C}_n(\Sigma,\partial^-(\Sigma)) ; V)
\end{equation}
for the relative Borel-Moore homology with coefficients in this local system. In this notation, \cite[\S 4.1]{HeisenbergHomology} explains that each mapping class $f \in \mathfrak{M}(\Sigma)$ induces an automorphism of $\mathcal{C}_n(\Sigma)$ covered by an isomorphism
\[
{}_{\tau \circ f_\Heis} \! V \longrightarrow {}_{\tau} \! V
\]
of local systems for each $\tau \in \mathrm{Aut}(\Heis)$. Taking relative Borel-Moore homology, we therefore obtain isomorphisms
\begin{equation}
\label{eq:action-of-f}
\mathcal{V}_n \bigl( {}_{\tau \circ f_\Heis} \! V \bigr) \longrightarrow \mathcal{V}_n \bigl( {}_{\tau} \! V \bigr)
\end{equation}
of $R$-modules. This may be described succinctly as a representation of the \emph{action groupoid} associated to the action \eqref{eq:action_on_Heis} of $\mathfrak{M}(\Sigma)$ on $\Heis$.

\begin{definition}
\label{def:action-groupoid}
For a group $G$ with a left action $a \colon G \to \mathrm{Sym}(X)$ on a set $X$, the \emph{action groupoid} $\mathrm{Ac}(G \curvearrowright X)$ is the groupoid whose set of objects is $a(G)$, whose set of morphisms $\sigma \to \tau$ is the subset $a^{-1}(\tau^{-1}\sigma) \subseteq G$ and whose composition is given by multiplication in $G$.
\end{definition}

\begin{theorem}[{\cite[Theorem~A(b)]{HeisenbergHomology}}]
\label{thm:twisted-representation}
Associated to any representation $V$ of $\Heis$ over $R$ and any integer $n\geq 2$, there is a functor 
\begin{equation}
\label{eq:twisted-representation}
\mathrm{Ac}(\mathfrak{M}(\Sigma) \curvearrowright \Heis) \longrightarrow \mathrm{Mod}_R
\end{equation}
sending each object $\tau \colon \Heis \to \Heis$ to the $R$-module $\mathcal{V}_n({}_{\tau} \! V)$ and sending each morphism $f \colon \tau \circ f_\Heis \to \tau$ to the $R$-linear isomorphism \eqref{eq:action-of-f}.
\end{theorem}

\begin{remark}
\label{rmk:untwisting}
The basic strategy to upgrade this to an \emph{untwisted} representation is to try to construct coefficient isomorphisms $V \cong {}_{f_\Heis} \! V$ for each $f$. Given this, one may then pre-compose \eqref{eq:action-of-f} with the induced isomorphism of twisted homology groups to obtain \emph{automorphisms} of $\mathcal{V}_n(V)$. We explained how to do this in \cite[\S 4.2]{HeisenbergHomology} when $V$ is the linearisation $L = \Heis \oplus \Z$ of the (affine) translation action of $\Heis$ on itself. We also explained in \cite[\S 5]{HeisenbergHomology} how to do this -- after passing to a certain central extension of $\mathfrak{M}(\Sigma)$ -- when $V$ is the Schrödinger representation of $\Heis$ (this is recalled briefly in \S\ref{subsec:untwisting-Schroedinger}). The goal of the present paper is to explain how to untwist on the Torelli group $\mathfrak{T}(\Sigma) \subset \mathfrak{M}(\Sigma)$ for \emph{any} representation $V$ of $\Heis$, as well as how to untwist on Earle-Morita subgroups of $\mathfrak{M}(\Sigma)$ when $V$ is the Schrödinger representation of $\Heis$ (without passing to any central extension).
\end{remark}

\section{Action on the Heisenberg group}
\label{s:Chillingworth}

The goal of this section is to study the action \eqref{eq:action_on_Heis} of the mapping class group $\mathfrak{M}(\Sigma)$ on the Heisenberg group $\Heis = \Heis(\Sigma)$, as well as the crossed homomorphism naturally associated to this action.

\subsection{Automorphisms of the Heisenberg group.}
\label{subsec:automorphisms-of-Heis}

As a first step, we recall a semi-direct product decomposition of an index-$2$ subgroup of the automorphism group $\mathrm{Aut}(\Heis)$.

Since the element $u=(1,0)$ of $\Heis$ generates its centre, which is infinite cyclic, any automorphism of $\Heis$ must send it either to itself or its inverse. We denote the group of automorphisms of $\Heis$ that fix $u$ by $\mathrm{Aut}^+(\Heis)$.
The structure of this group was studied in \cite[\S 3]{HeisenbergHomology}. There is a natural homomorphism $\cL \colon \mathrm{Aut}^+(\Heis) \to Sp(H)$, where we write $H = H_1(\Sigma;\Z)$, and a split short exact sequence
\begin{equation}
\label{eq:Aut-Heis-ses}
\begin{tikzcd}
1 \ar[r] & H^1(\Sigma;\Z) \ar[r,"j"] & \mathrm{Aut}^+(\Heis) \ar[r,"\cL"] & Sp(H) \ar[r] & 1
\end{tikzcd}
\end{equation}
where $j(c)=[(k,x)\mapsto (k+c(x),x)]$.
The splitting gives a decomposition $\mathrm{Aut}^+(\Heis) \cong Sp(H) \ltimes H^1(\Sigma;\Z)$, where the semi-direct product structure on the right-hand side is induced by the natural action of $Sp(H)$ on $\mathrm{Hom}(H,\Z) \cong H^1(\Sigma;\Z)$.
The projection onto the right-hand factor of this decomposition is a function $(-)^\diamond \colon \mathrm{Aut}^+(\Heis) \to H^1(\Sigma;\Z) \cong \mathrm{Hom}(H,\Z)$ (which is not a group homomorphism) given by the assignment $\varphi \mapsto \varphi^\diamond = \mathrm{pr}_1(\varphi(0,-))$.

For a mapping class $f\in \mathfrak{M}(\Sigma)$, the map $f_\Heis$ of \eqref{eq:projection-equivariance} is represented as follows:
\begin{equation}\label{eq:action}
f_\Heis \colon (k,x)\mapsto (k+\delta_f(x),f_*(x)),
\end{equation}
where $\delta_f = (f_\Heis)^\diamond \in H^1(\Sigma;\Z) \cong \mathrm{Hom}(H,\Z)$.
We proved in \cite[\S 3]{HeisenbergHomology} that the map $\delta \colon \mathfrak{M}(\Sigma) \to H^1(\Sigma;\Z)$ given by $f\mapsto \delta_f$ is a crossed homomorphism, meaning that
\[
\delta_{g\circ f}(x) = \delta_f(x)+f^*(\delta_g)(x)\ .
\]
We also identified this crossed homomorphism, showing that it coincides with the combinatorially-defined crossed homomorphism $\mathfrak{d}$ constructed by Morita \cite{Morita1989}, who proved that it represents a generator of $H^1(\mathfrak{M}(\Sigma);H^1(\Sigma;\Z))$, which is infinite cyclic. In fact, a different crossed homomorphism $\psi$ had been constructed somewhat earlier by Earle \cite{Earle1978}, and turned out, in light of \cite{Morita1989}, also to represent a generator of $H^1(\mathfrak{M}(\Sigma);H^1(\Sigma;\Z))$. The precise relationship between the crossed homomorphisms $\mathfrak{d}$ and $\psi$ was elucidated in \cite{Kuno2009}. In \S\ref{Trapp-representation} below, we discuss the relationship of $\mathfrak{d} = \delta$ to the \emph{Trapp representation} \cite{Trapp} (see Proposition \ref{prop:identification_crossed_hom}).

\begin{definition}
\label{def:Earle-Morita}
The \emph{Earle-Morita subgroup} $\mathrm{Mor}(\Sigma) \subseteq \mathfrak{M}(\Sigma)$ is defined to be the kernel of $\mathfrak{d}$.
\end{definition}

\begin{remark}
\label{rmk:Earle-Morita-choices}
The Earle-Morita subgroup is not a normal subgroup of $\mathfrak{M}(\Sigma)$, despite being a kernel; this is because it is a kernel of a \emph{crossed} homomorphism. We remark also that there are many Earle-Morita subgroups, since the definition of $\mathfrak{d}$ (or, equivalently, the definition of $\delta$) depends on a choice. The definition of $\delta$ above uses the splitting of the short exact sequence \eqref{eq:Aut-Heis-ses}. Note that it depends on the choice of isomorphism \eqref{eq:Heisenberg-quotient-isomorphism} in Proposition \ref{hom_phi}, identifying the relevant quotient of the surface braid group with (an explicit model for) the Heisenberg group.
\end{remark}

\subsection{The Chillingworth subgroup}
\label{subsec:Chillingworth}

Recall that the \emph{Torelli subgroup} $\mathfrak{T}(\Sigma) \subseteq \mathfrak{M}(\Sigma)$ consists of those elements of the mapping class group whose natural action on $H_1(\Sigma;\Z)$ is trivial. The restriction of the crossed homomorphism $\delta \colon f\mapsto \delta_f$ to the Torelli group is a homomorphism. 
We will first describe this homomorphism in relation with the action of the Torelli group on homotopy classes of vector fields. Recall that the set $\Xi(\Sigma)$ of homotopy classes of non-vanishing vector fields on $\Sigma$ supports a natural simply transitive action of $H^1(\Sigma;\Z)$ (in other words, an affine structure over $\Z$ with associated $\Z$-module $H^1(\Sigma;\Z)$), and the action of $\mathfrak{M}(\Sigma)$ is compatible with this action. It follows that the Torelli group acts by translation on $\Xi(\Sigma)$, which defines a homomorphism $e \colon \mathfrak{T}(\Sigma)\rightarrow H^1(\Sigma;\Z)$. A formula for $e(f)([\gamma])$, where $\gamma$ is a regular curve, is given by the variation of the winding number. For convenience we recall some details about the winding number below.

Fix a Riemannian metric on $\Sigma$. A non-vanishing vector field $X$ gives a trivialisation of the unit tangent bundle $T_1(\Sigma)\cong \Sigma\times S^1$. The winding number $\omega_X(\gamma)$ of a regular oriented curve $\gamma$ is the degree of the second component of the unit tangent vector. It can be computed as follows.
Assuming that $\gamma$ is transverse to $X$ except at a finite set $\gamma \pitchfork X$ of points, where it looks locally as in Figure \ref{fig:winding-sign}, then
\[
\omega_X(\gamma) = \sum_{p \in \gamma \pitchfork X} \mathrm{sgn}(p),
\]
where $\mathrm{sgn}(p)$ is defined in Figure \ref{fig:winding-sign}. Notice that only the points $p$ where the tangent vector of $\gamma$ is \emph{parallel} to $X$ count towards this sum; those that point in the opposite direction to $X$ do not.

\begin{figure}[t]
\centering
\includegraphics[scale=0.8]{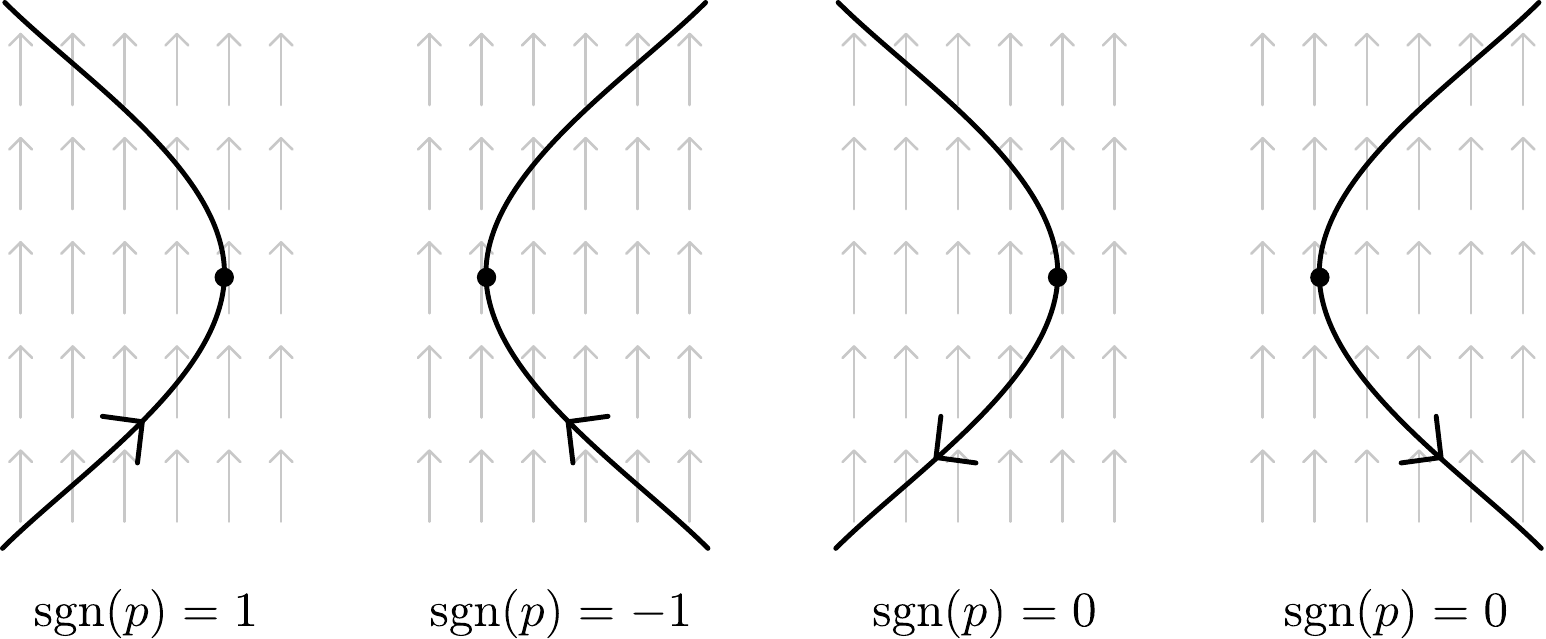}
\caption{The sign of a point on $\gamma$ that is tangent to $X$.}
\label{fig:winding-sign}
\end{figure}

\begin{definition}
The \emph{Chillingworth homomorphism} $e \colon \mathfrak{T}(\Sigma)\rightarrow H^1(\Sigma;\Z)$, studied in \cite{Chillingworth,Johnson}, is defined by
\begin{equation}
\label{eq:Chillingworth-formula}
e(f)([\gamma]) = \omega_X(f \circ \gamma) - \omega_X(\gamma)\ .
\end{equation}
Its kernel is the \emph{Chillingworth subgroup} $\mathrm{Chill}(\Sigma)$. We note that $e$ does not depend on the choice of non-vanishing vector field $X$, but extends to a crossed homomorphism $e_X \colon \mathfrak{M}(\Sigma)\rightarrow H^1(\Sigma;\Z)$ that does, as we shall discuss in \S\ref{Trapp-representation} (see Definition \ref{def:Trapp}).
\end{definition}

\begin{remark}
Since we have $\delta = \mathfrak{d}$ (as recalled in \S\ref{subsec:automorphisms-of-Heis} above) and $e = \delta$ on the Torelli group (Lemma \ref{lem:coincide} below), we have:
\[
\mathrm{Chill}(\Sigma) = \mathrm{ker}(e) = \mathfrak{T}(\Sigma) \cap \mathrm{ker}(\mathfrak{d}) = \mathfrak{T}(\Sigma) \cap \mathrm{Mor}(\Sigma).
\]
Equivalently, we may say that $\mathrm{Chill}(\Sigma)$ is the intersection of the kernels of $\mathfrak{s}$ and $\mathfrak{d}$, in other words it is the kernel of $(\mathfrak{s},\mathfrak{d}) \colon \mathfrak{M}(\Sigma) \to Sp(H) \ltimes H$. Notice in particular that, although the Earle-Morita subgroup $\mathrm{Mor}(\Sigma)$ depends on a non-canonical choice (see Remark \ref{rmk:Earle-Morita-choices}), its intersection with the Torelli group does not.
\end{remark}

The following lemma is Proposition 3.7 of \cite{Brendlethesis}. The proof there uses a result of \cite{Morita1993}. We give an independent proof below.

\begin{lemma}
\label{lem:coincide}
The homomorphisms $\delta$ and $e$ coincide on the Torelli group and have image $\delta(\mathfrak{T}(\Sigma) ) = 2.H^1(\Sigma;\Z)$.
\end{lemma}

From formula \eqref{eq:action} it follows that the kernel of the action 
$\Psi \colon \mathfrak{M}(\Sigma)\rightarrow \mathrm{Aut}^+(\Heis)$ is contained in the Torelli group; we may therefore identify this kernel as a corollary of Lemma \ref{lem:coincide}.

\begin{proposition}
\label{kernel_of_Psi}
For any genus $g \geq 1$, we have $\mathrm{ker}(\Psi) = \mathrm{Chill}(\Sigma)$.
\end{proposition}
\begin{proof}
From formula \eqref{eq:action} we see that $\mathrm{ker}(\Psi) = \mathfrak{T}(\Sigma) \cap \mathrm{ker}(\delta)$; by Lemma \ref{lem:coincide} this is equal to $\mathrm{ker}(e) = \mathrm{Chill}(\Sigma)$.
\end{proof}

Denote by $\mathrm{Inn}(\Heis)$ the group of \emph{inner} automorphisms of the Heisenberg group $\Heis$. From Lemma \ref{lem:coincide}, we may also identify the \emph{projective kernel} of the action $\Psi$, namely the subgroup $\Psi^{-1}(\mathrm{Inn}(\Heis))$ that acts by inner automorphisms.

\begin{proposition}
\label{Psi_inner}
For any genus $g \geq 1$, we have $\Psi^{-1}(\mathrm{Inn}(\Heis)) = \mathfrak{T}(\Sigma)$.
\end{proposition}
\begin{proof}
Conjugation in the Heisenberg group $\Heis$ is given by the  formula
\begin{equation} \label{eq:inner}
(l,x)(k,y)(-l,-x)=(l,y)(k,x)(-l,-x)=(k+2x.y , y)
\end{equation}

First, if $\Psi(f) = f_\Heis$ is an inner automorphism, then its induced action on $H$ must be trivial. This means that $f$  lies in the Torelli group. Conversely, if $f \in \mathfrak{T}(\Sigma)$, we have from Lemma \ref{lem:coincide} that $\delta_f$ is in $2.H^1(\Sigma;\Z)$.
Using Poincaré duality, we obtain $x\in H$ such that $\delta_f(y)=2x.y$ for every $y \in H$. Comparing formulas \eqref{eq:action} and \eqref{eq:inner}, we deduce that $f_\Heis$ is inner.
\end{proof}

We will use the following basic lemma about crossed homomorphisms in the proof of Lemma \ref{lem:coincide}.

\begin{lemma}
\label{lem:crossed-hom-agree}
Let $G$ be a group acting on an abelian group $K$ and denote by $N \subseteq G$ the kernel of this action. Suppose that $S \subseteq N$ normally generates $N$ in $G$. If two crossed homomorphisms $\theta_1,\theta_2 \colon G \to K$ agree on $S$, then they agree on $N$.
\end{lemma}
\begin{proof}
The assumption that $S$ normally generates $N$ in $G$ means that
\[
T = \{ gsg^{-1} \mid s \in S, g \in G \} \subseteq N
\]
generates $N$. It will therefore suffice to show that $\theta_1$ and $\theta_2$ agree on $T$. Let $s\in S$ and $g\in G$. We know by hypothesis that $\theta_1(s) = \theta_2(s)$, and we need to show that $\theta_1(g^{-1}sg) = \theta_2(g^{-1}sg)$. First, observe that, for $i=1,2$, we have
\[
\theta_i(g) + g.\theta_i(g^{-1}) = \theta_i(g^{-1}g) = \theta_i(1) = 0.
\]
Using this, and the fact that $s\in N$, so it acts trivially on $K$, we deduce that
\begin{align*}
\theta_i(g^{-1}sg) &= \theta_i(g) + g.\theta_i(s) + gs.\theta_i(g^{-1}) \\
&= \theta_i(g) + g.\theta_i(s) + g.\theta_i(g^{-1}) \\
&= g.\theta_i(s).
\end{align*}
Thus $\theta_1(g^{-1}sg) = g.\theta_1(s) = g.\theta_2(s) = \theta_2(g^{-1}sg)$, as required.
\end{proof}

\begin{proof}[Proof of Lemma \ref{lem:coincide}]
The Torelli group is generated by genus-one bounding pair diffeomorphisms \cite[Theorem 2]{Johnson79}, and this generating set is a single conjugacy class in the full mapping class group. It follows that the Torelli group is normally generated by a(ny) single genus-one bounding pair diffeomorphism $f$. We may therefore apply Lemma \ref{lem:crossed-hom-agree} to the setting where $G = \mathfrak{M}(\Sigma)$, $K = H^1(\Sigma;\Z)$, $N = \mathfrak{T}(\Sigma)$ and $S = \{f\}$. Both $\delta$ and $e$ extend to crossed homomorphisms defined on the full mapping class group (for $e$ such an extension is given by the Trapp representation, as recalled in \S\ref{Trapp-representation} below). To prove that $\delta$ and $e$ coincide on $\mathfrak{T}(\Sigma)$, it is therefore sufficient to show that $\delta$ and $e$ agree on the single element $f$. Moreover, to show that $\delta(\mathfrak{T}(\Sigma)) \subseteq 2.H^1(\Sigma;\Z)$ it is enough to show that this common value $\delta_f = e(f)$ lies in $2.H^1(\Sigma;\Z)$. Specifically, we will take this element to be
\[
f = BP(\gamma,\delta) = T_\gamma . T_\delta^{-1},
\]
the genus one bounding pair diffeomorphism depicted in Figure~\ref{fig:bounding-pair}, and we will show that the elements $e(f)$ and $\delta_f$ of $H^1(\Sigma;\Z) \cong \mathrm{Hom}(H,\Z)$ are both equal to the homomorphism $H = H_1(\Sigma;\Z) \to \Z$ given by
\begin{equation}
\label{eq:evaluation-of-f}
a_1 \mapsto 2 \qquad\text{,}\qquad a_i \mapsto 0 \text{ for } i \geq 2 \qquad\text{and}\qquad b_i \mapsto 0 \text{ for } i \geq 1.
\end{equation}
As explained just above, this calculation will show that $\delta = e$ on $\mathfrak{T}(\Sigma)$ and that $\delta(\mathfrak{T}(\Sigma)) \subseteq 2.H^1(\Sigma;\Z)$. The opposite inclusion will also follow, since $\delta_f = \eqref{eq:evaluation-of-f}$ is a free generator of $2.H^1(\Sigma;\Z)$ and the other $2g-1$ elements of the evident free generating set may be realised similarly as $\delta_{f'}$ for analogous elements $f'$.

\begin{figure}[ht]
\centering
\includegraphics[scale=0.6]{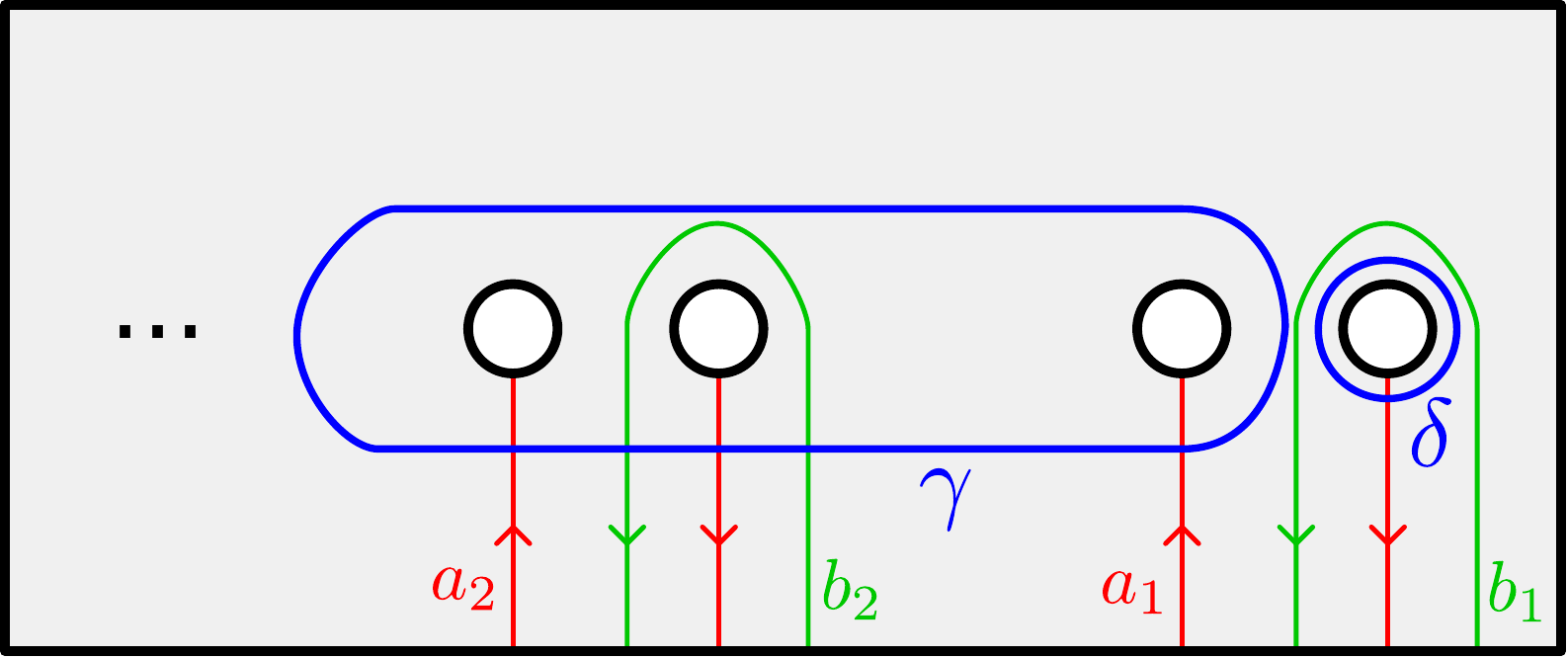}
\caption{The surface $\Sigma$ is obtained by identifying the $2g$ interior boundary components (four of which are depicted above) in $g$ pairs by reflections. The bounding pair map from the proof of Lemma \ref{lem:coincide} is $BP(\gamma,\delta) = T_\gamma . T_\delta^{-1}$, for the blue curves $\gamma$ and $\delta$. The red and green arcs form a symplectic basis for the first homology of $\Sigma$ relative to the bottom edge $\partial^-(\Sigma)$.}
\label{fig:bounding-pair}
\end{figure}

It therefore remains to calculate that $\delta_f = e(f) = \eqref{eq:evaluation-of-f}$.

We first calculate $\delta_f$ from the automorphism $f_\Heis$. We may directly read off from Figure \ref{fig:bounding-pair} the effect of $f_\Heis$ on the elements $\tilde a_i$ and $\tilde b_i$ of $\Heis$. It clearly acts trivially except possibly on the three elements $\tilde a_2=(0,a_2)$, $\tilde b_2=(0,b_2)$ and $\tilde a_1=(0,a_1)$, since the others may be realised disjointly from $\gamma \cup \delta$, and:
\begin{align*}
\tilde a_1 &\mapsto [\tilde a_2,\tilde b_2].\tilde a_1 = u^2 \tilde a_1=(2,a_1) \\
\tilde a_2 &\mapsto [\tilde a_2,\tilde b_2].\tilde a_1 \tilde b_1 \tilde a_1^{-1}. \tilde a_2. \tilde a_1 \tilde b_1^{-1} \tilde a_1^{-1}. [\tilde a_2,\tilde b_2]^{-1} = \tilde a_2 \\
\tilde b_2 &\mapsto [\tilde a_2,\tilde b_2].\tilde a_1 \tilde b_1 \tilde a_1^{-1}. \tilde b_2. \tilde a_1 \tilde b_1^{-1} \tilde a_1^{-1}. [\tilde a_2,\tilde b_2]^{-1} = \tilde b_2.
\end{align*}
This gives the calculation $\delta_f = \eqref{eq:evaluation-of-f}$.

\begin{figure}[ht]
\centering
\includegraphics[scale=0.8]{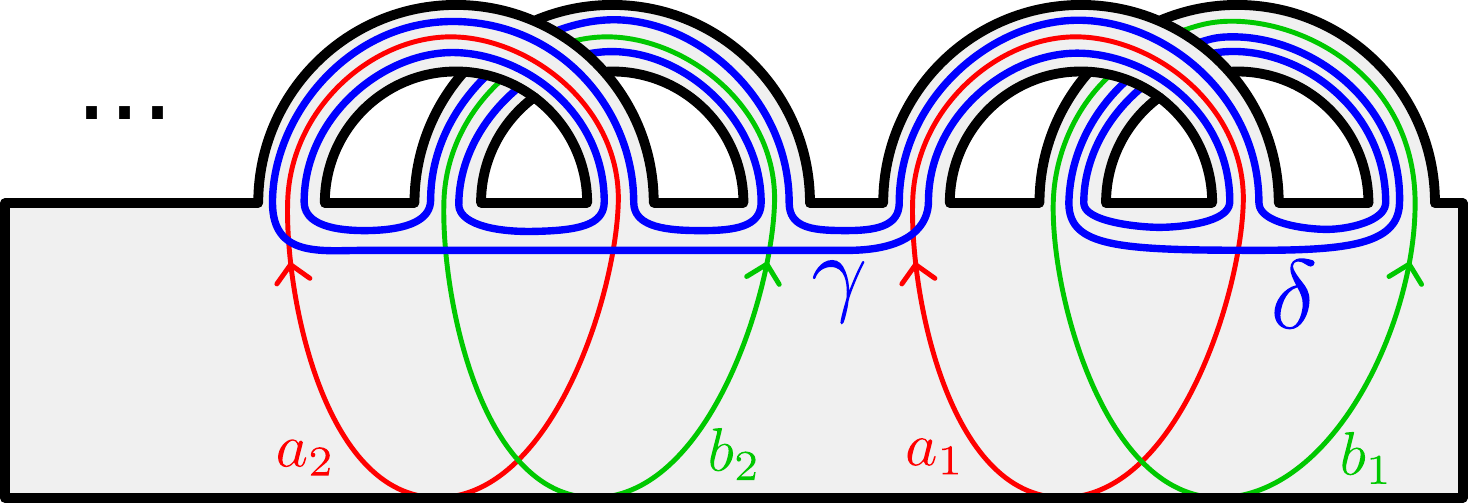}
\caption{An alternative model for the surface $\Sigma$, the bounding pair $(\gamma,\delta)$ and the symplectic basis for the first homology of $\Sigma$ relative to the bottom edge $\partial^-(\Sigma)$.}
\label{fig:bounding-pair-2}
\end{figure}

To calculate $e(f)$, we use the alternative model for the surface $\Sigma$, the bounding pair $(\gamma,\delta)$ and the symplectic basis $a_i,b_i$ for $H$ depicted in Figure \ref{fig:bounding-pair-2}. This model for $\Sigma$ has the advantage of having an obvious non-vanishing vector field $X$, which simply points \emph{upwards} according to the standard framing of the page.

Using this vector field $X$ and comparing to Figure \ref{fig:winding-sign}, we observe that the winding numbers of the symplectic generators $a_i$ and $b_i$ (more precisely, their smooth, closed representatives pictured in Figure \ref{fig:bounding-pair-2}) are given by
\[
\omega_X(a_i) = -1 \qquad\text{and}\qquad \omega_X(b_i) = +1.
\]
We recall that, by definition, $e(f)(c) = \omega_X(f \circ \bar{c}) - \omega_X(\bar{c}) \in \Z$ for any $c = [\bar{c}] \in H$. We clearly have $f\circ \bar{c} = \bar{c}$ for $\bar{c} = a_i \text{ or } b_i$ with $i \geq 3$ or for $\bar{c} = b_1$, since these curves may be represented disjointly from $\gamma \cup \delta$. Hence $e(f)([\bar{c}]) = 0$ for these $\bar{c}$.

The curve $f \circ a_1$ is depicted in Figure \ref{fig:fa1}.

\begin{figure}[ht]
\centering
\includegraphics[scale=0.8]{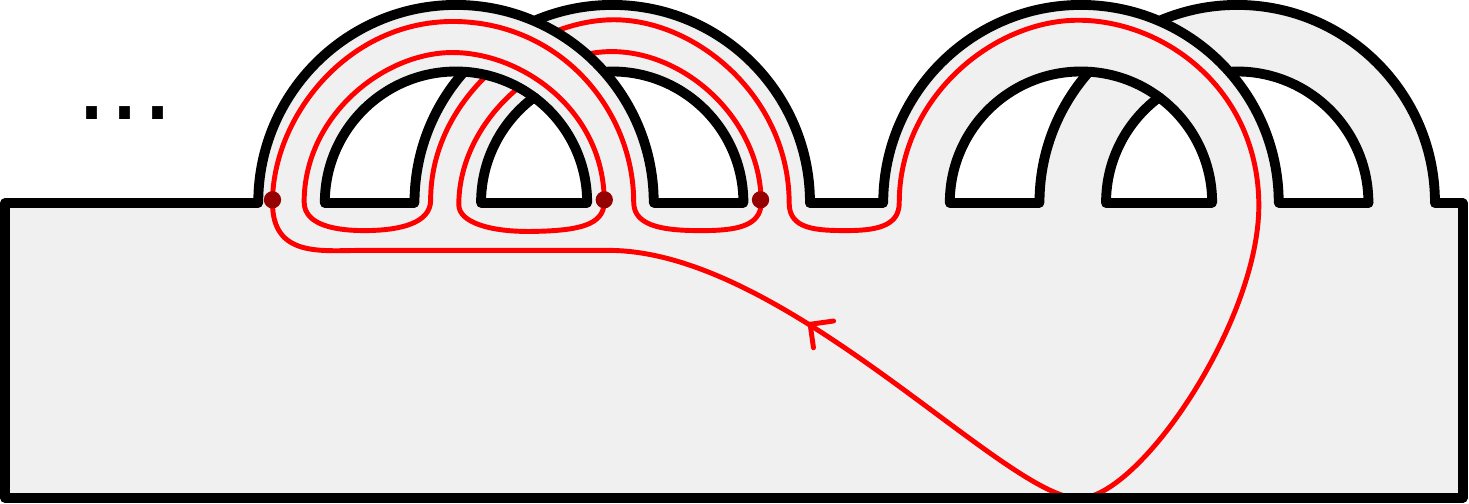}
\caption{The curve $f \circ a_1$ for $f = T_\gamma . T_\delta^{-1}$. The three points where its tangent vector points vertically upwards are marked with dark red points: the left-most one is negative according to Figure \ref{fig:winding-sign}, and the other two are positive. {\small [At first sight it may look like there are two more, but these are not allowed since they do not fit either of the local models of Figure \ref{fig:winding-sign}. We therefore perturb the curve slightly to get rid of these two tangencies with the vector field $X$. Alternatively, we may perturb it differently, to turn each of these disallowed tangencies into a pair of two allowed tangencies with opposite signs, which will therefore cancel in the expression for $\omega_X(f\circ a_1)$.]}}
\label{fig:fa1}
\end{figure}

There are precisely three points on this curve where its tangent vector is equal to the vector field $X$, i.e., where its tangent vector is pointing vertically upwards: two are positive and one is negative (compare the local models in Figure \ref{fig:winding-sign}), hence
\begin{align*}
e(f)(a_1) &= \omega_X(f \circ a_1) - \omega_X(a_1) \\
&= (2-1) - (-1) \\
&= 2.
\end{align*}
Now let $\bar{c}$ be either $a_2$ or $b_2$. In this case the effect of $f$ is simply to conjugate $\bar{c}$ by $\gamma$, so we have that
\begin{align*}
\omega_X(f \circ \bar{c}) &= \omega_X(\gamma) + \omega_X(\bar{c}) - \omega_X(\gamma) \\
&= \omega_X(\bar{c}),
\end{align*}
since positive/negative tangencies with $X$ for $\gamma$ are negative/positive tangencies with $X$ for $\gamma^{-1}$ respectively, and so $e(f)([\bar{c}]) = \omega_X(f \circ \bar{c}) - \omega_X(\bar{c}) = 0$. Thus we have shown that $e(f) \colon H \to \Z$ is also given by \eqref{eq:evaluation-of-f}.
\end{proof}

\subsection{The Trapp representation.}
\label{Trapp-representation}

We next recall the \emph{Trapp representation} \cite{Trapp}, and show that our representation \eqref{eq:action_on_Heis} of $\mathfrak{M}(\Sigma)$ on $\Heis$ may be identified with it, up to ``coboundaries'', when the genus $g$ of $\Sigma$ is at least $2$. This provides an alternative proof of Proposition \ref{kernel_of_Psi} (except when $g=1$), since the kernel of the Trapp representation is equal to the Chillingworth subgroup $\mathrm{Chill}(\Sigma)$ when $g \geq 2$~\cite[Corollary 2.7]{Trapp}.

\begin{definition}
\label{def:Trapp}
Write $H = H_1(\Sigma;\Z)$ and $H^* = \mathrm{Hom}(H,\Z) \cong H^1(\Sigma;\Z)$. The representation of Trapp~\cite{Trapp} is defined as a homomorphism
\begin{equation}
\label{eq:Trapp}
\Phi_X \colon \mathfrak{M}(\Sigma) \longrightarrow Sp(H) \ltimes H^* \subset GL_{2g+1}(\Z)
\end{equation}
lifting the standard symplectic action $\mathfrak{s} \colon \mathfrak{M}(\Sigma) \to Sp(H)$. Having fixed this choice of symplectic action, the homomorphism \eqref{eq:Trapp} corresponds to a choice of crossed homomorphism
\begin{equation}
\label{eq:Trapp-crossed-hom}
e_X \colon \mathfrak{M}(\Sigma) \longrightarrow H^* .
\end{equation}
This crossed homomorphism is given by the variation of the winding number with respect to a fixed non-vanishing vector field $X$ on $\Sigma$, as already discussed in \S\ref{subsec:Chillingworth}; see the formula \eqref{eq:Chillingworth-formula}.
\end{definition}

We therefore have two homomorphisms
\[
\Psi = \eqref{eq:action_on_Heis} \text{ and } \Phi_X = \eqref{eq:Trapp} \colon \mathfrak{M}(\Sigma) \longrightarrow Sp(H) \ltimes H^*
\]
corresponding to crossed homomorphisms $\delta \text{ and } e_X \colon \mathfrak{M}(\Sigma) \to H^*$. We have proven above (Lemma \ref{lem:coincide}) that these crossed homomorphisms are equal when restricted to the Torelli group. We now strengthen this to show that $\delta$ and $e_X$ agree, modulo coboundaries, on the whole mapping class group.

\begin{proposition}
\label{prop:identification_crossed_hom}
For $g\geq 2$, the crossed homomorphisms $\delta$ and $e_X$ represent the same cohomology class in $H^1(\mathfrak{M}(\Sigma);H^*)\cong \Z$. In other words, they are equal modulo principal crossed homomorphisms, i.e.~coboundaries.
\end{proposition}
\begin{proof}
We will use the homomorphism
\begin{equation}
\label{eq:cohomology-to-Hom}
H^1(\mathfrak{M}(\Sigma);H^*) \longrightarrow \mathrm{Hom}(\mathfrak{T}(\Sigma),H^*)
\end{equation}
given by restricting a crossed homomorphism $\mathfrak{M}(\Sigma) \to H^*$ to the Torelli group. This is well-defined since principal crossed homomorphisms (coboundaries) are trivial on the Torelli group. The right-hand side of \eqref{eq:cohomology-to-Hom} is rather large: by a theorem of Johnson \cite{Johnson85}, the abelianisation of $\mathfrak{T}(\Sigma)$ is isomorphic to $\wedge^3 H \oplus \text{(torsion)}$, so $\mathrm{Hom}(\mathfrak{T}(\Sigma),H^*) \cong \mathrm{Hom}(\wedge^3 H,H^*)$, which is free abelian of rank $2g\binom{2g}{3}$. However, it has the advantage that it is easy to detect when its elements are equal, since it is just a group of homomorphisms (rather than crossed homomorphisms modulo principal ones). On the other hand, the left-hand side of \eqref{eq:cohomology-to-Hom} is much smaller. Indeed, Morita proved in \cite[Proposition 6.4]{Morita1989} that the group $H^1(\mathfrak{M}(\Sigma);H^*)$ is infinite cyclic. (In fact, it is generated by $[\mathfrak{d}]$, which we showed in \cite{HeisenbergHomology} is equal to $[\delta]$, but we will not need this.) In Lemma \ref{lem:coincide} we have proven that $\delta$ and $e_X$ coincide, and are non-trivial, on the Torelli subgroup. Since $\mathrm{Hom}(\mathfrak{T}(\Sigma),H^*)$ is torsion-free, the homomorphism \eqref{eq:cohomology-to-Hom} is injective and the result follows.
\end{proof}

\begin{remark}
In summary, we have considered three crossed homomorphisms
\[
\delta , \mathfrak{d} , e_X \colon \mathfrak{M}(\Sigma) \longrightarrow H^* \cong H^1(\Sigma;\Z),
\]
where $\delta$ is the crossed homomorphism corresponding to the action \eqref{eq:action_on_Heis} of the mapping class group on the Heisenberg group, $\mathfrak{d}$ is Morita's crossed homomorphism (whose precise relationship to Earle's crossed homomorphism $\psi$ is described in \cite{Kuno2009}) and $e_X$ is Chillingworth's crossed homomorphism, depending on a choice of non-vanishing vector field $X$ on $\Sigma$. We showed in \cite{HeisenbergHomology} that $\delta = \mathfrak{d}$ on $\mathfrak{M}(\Sigma)$. In this section, we have shown (Lemma \ref{lem:coincide}) that $\delta = e_X$ when restricted to $\mathfrak{T}(\Sigma)$ and, moreover, that $\delta = e_X$ on $\mathfrak{M}(\Sigma)$ modulo coboundaries (Proposition \ref{prop:identification_crossed_hom}). We note, however, that only the weaker statement of Lemma \ref{lem:coincide} was needed to deduce (Propositions \ref{kernel_of_Psi} and \ref{Psi_inner}) that the kernel of $\Psi$ is  $\mathrm{Chill}(\Sigma)$ and the projective kernel of $\Psi$ is $\mathfrak{T}(\Sigma)$.
\end{remark}

\subsection{Restricting to the Chillingworth subgroup.}
\label{rep-Chillingworth}

Using Proposition \ref{kernel_of_Psi}, we deduce that the twisted representations of $\mathfrak{M}(\Sigma)$ constructed in Theorem \ref{thm:twisted-representation} are in fact untwisted when restricted to $\mathrm{Chill}(\Sigma) \subset \mathfrak{M}(\Sigma)$.

\begin{theorem}
\label{thm:Chill}
Associated to any representation $V$ of $\Heis$ over $R$ and any integer $n\geq 2$, there is a representation
\begin{equation}
\label{eq:Chillingworth-representation}
\mathrm{Chill}(\Sigma) \longrightarrow \mathrm{Aut}_R \bigl( H_n^{BM}\bigl(\mathcal{C}_{n}(\Sigma),\mathcal{C}_{n}(\Sigma,\partial^-(\Sigma)) ;V\bigr)\bigr) ,
\end{equation}
which is a restriction of \eqref{eq:twisted-representation} to a single object of the action groupoid.
\end{theorem}
\begin{proof}
This follows from the construction described in \S\ref{subsec:twisted-representations}, with each element $f \in \mathrm{Chill}(\Sigma)$ acting by the automorphism \eqref{eq:action-of-f} (setting $\tau = \mathrm{id}$), since we know from Proposition \ref{kernel_of_Psi} that $f_\Heis = \Psi(f) = \mathrm{id}$ for each $f \in \mathrm{Chill}(\Sigma)$.
\end{proof}

\section{Representations of the Torelli group}
\label{s:Torelli}

We now restrict to the Torelli group $\mathfrak{T}(\Sigma) \subseteq \mathfrak{M}(\Sigma)$. By Proposition \ref{Psi_inner}, the Torelli group acts by inner automorphisms under $\Psi$, so we have a homomorphism
\begin{equation}
\label{eq:Psi-restricted-to-Torelli}
\Psi \colon \mathfrak{T}(\Sigma) \longrightarrow \mathrm{Inn}(\Heis) \subset \mathrm{Aut}(\Heis).
\end{equation}
We may therefore pull back the $\Z$-central extension
\begin{equation}
\label{eq:Heisenberg-extension}
1 \to \Z \cong \mathcal{Z}(\Heis) \longrightarrow \Heis \longrightarrow \mathrm{Inn}(\Heis) \to 1
\end{equation}
along \eqref{eq:Psi-restricted-to-Torelli} to obtain a $\Z$-central extension
\begin{equation}
\label{eq:Torelli-extension}
1 \to \Z \longrightarrow \widetilde{\mathfrak{T}}(\Sigma) \longrightarrow \mathfrak{T}(\Sigma) \to 1
\end{equation}
and a homomorphism
\begin{equation}
\label{eq:Psi-tilde}
\widetilde{\Psi} \colon \widetilde{\mathfrak{T}}(\Sigma) \longrightarrow \Heis
\end{equation}
lifting \eqref{eq:Psi-restricted-to-Torelli}.

\begin{remark}
\label{rmk:cocycle-for-ext-of-Torelli}
The inner automorphism group $\mathrm{Inn}(\Heis)$ naturally identifies with the first homology group $H = H_1(\Sigma;\Z)$. Under this identification, \eqref{eq:Psi-restricted-to-Torelli} becomes the crossed homomorphism $\delta \colon \mathfrak{T}(\Sigma) \to H^*$ (which is a \emph{homomorphism} on the Torelli group) composed with the Poincaré duality isomorphism $H^* \cong H$. A $2$-cocycle representing the extension \eqref{eq:Heisenberg-extension} is the intersection form $.$ on $H$. A $2$-cocycle representing \eqref{eq:Torelli-extension} is therefore given by $(f,f') \mapsto \delta(f)^\sharp . \delta(f')^\sharp$, where $(\phantom{-})^\sharp$ denotes Poincaré duality. As mentioned above, we showed in \cite{HeisenbergHomology} that $\delta = \mathfrak{d}$, so this $2$-cocycle may also be written as $(f,f') \mapsto \mathfrak{d}(f)^\sharp . \mathfrak{d}(f')^\sharp$.
\end{remark}

Recall from \S\ref{subsec:review} that, for any representation $V$ of $\Heis$, the Heisenberg homology module $H_n^{BM}\bigl(\mathcal{C}_{n}(\Sigma),\mathcal{C}_{n}(\Sigma,\partial^-(\Sigma)) ;V\bigr)$ is obtained from the singular chain complex $\mathcal{S}_*(\widetilde{\mathcal{C}}_n(\Sigma))$ of the Heisenberg covering $\widetilde{\mathcal{C}}_n(\Sigma)$ by:
\begin{itemize}
\item tensoring over $\Z[\Heis]$ with $V$ \eqref{eq:Local};
\item taking the quotient of this complex by the subcomplex corresponding to the subspace $\mathcal{C}_{n}(\Sigma,\partial^-(\Sigma)) \cup (\mathcal{C}_{n}(\Sigma)\setminus T)$, for each compact $T \subset \mathcal{C}_n(\Sigma)$;
\item passing to homology and then taking the inverse limit over $T$ \eqref{eq:relativeBorelMoore}.
\end{itemize}
The isomorphism \eqref{eq:action-of-f} (for $\tau = (f_\Heis)^{-1}$) is induced by the natural twisted action of $f \in \mathfrak{M}(\Sigma)$ on $\mathcal{S}_*(\widetilde{\mathcal{C}}_n(\Sigma))$, which is of the form
\begin{equation}
\label{eq:action-on-singular-complex}
\mathcal{S}_*(\widetilde{\mathcal{C}}_n(\Sigma)) \longrightarrow \mathcal{S}_*(\widetilde{\mathcal{C}}_n(\Sigma))_{f_\Heis}.
\end{equation}

Now, for an element $h \in \Heis$, let us denote by $c_h = h{-}h^{-1}$ the corresponding inner automorphism $c_h \in \mathrm{Inn}(\Heis)$. One may verify that the isomorphism
\begin{equation}
\label{eq:Torelli-step-1}
- \cdot h \colon \mathcal{S}_* (\widetilde{\mathcal{C}}_n(\Sigma))_{c_h} \longrightarrow \mathcal{S}_* (\widetilde{\mathcal{C}}_n(\Sigma))
\end{equation}
of singular chain complexes given by the right-action of $h$ is $\Z[\Heis]$-linear. For each $\widetilde{f} \in \widetilde{\mathfrak{T}}(\Sigma)$, we may then take the composition
\begin{equation}
\label{eq:Torelli-composed}
\begin{tikzcd}
\mathcal{S}_*(\widetilde{\mathcal{C}}_n(\Sigma)) \ar[r,"{\eqref{eq:action-on-singular-complex}}"] & \mathcal{S}_* (\widetilde{\mathcal{C}}_n(\Sigma))_{f_\Heis} \ar[r,"{\eqref{eq:Torelli-step-1}}"] & \mathcal{S}_*(\widetilde{\mathcal{C}}_n(\Sigma)) \ ,
\end{tikzcd}
\end{equation}
where $f$ denotes the projection of $\widetilde{f}$ to $\mathfrak{T}(\Sigma)$ and we set $h = \widetilde{\Psi}(\widetilde{f}) \in \Heis$. The fact that $f_\Heis = c_h$ follows from the fact that \eqref{eq:Psi-tilde} is a lift of \eqref{eq:Psi-restricted-to-Torelli}. This defines an untwisted, $\Z[\Heis]$-linear action of $\widetilde{\mathfrak{T}}(\Sigma)$ on the singular chain complex $\mathcal{S}_*(\widetilde{\mathcal{C}}_n(\Sigma))$. By the construction recalled above, this in turn induces an untwisted, $R$-linear action of $\widetilde{\mathfrak{T}}(\Sigma)$ on Heisenberg homology:
\begin{equation}
\label{eq:representation-Torelli-extension}
\widetilde{\mathfrak{T}}(\Sigma) \longrightarrow \mathrm{Aut}_R \bigl( H_n^{BM} \bigl( \mathcal{C}_n(\Sigma) , \mathcal{C}_{n}(\Sigma,\partial^-(\Sigma)) ; V \bigr) \bigr) .
\end{equation}

To complete the construction, we show that:

\begin{lemma}
\label{lem:trivial-extension-of-Torelli}
The central extension $\widetilde{\mathfrak{T}}(\Sigma)$ of $\mathfrak{T}(\Sigma)$ is trivial, i.e.~it is isomorphic to the product $\mathfrak{T}(\Sigma) \times \Z$.
\end{lemma}

\begin{theorem}
\label{thm:Torelli}
Associated to any representation $V$ of $\Heis$ over $R$ and any integer $n\geq 2$, there is a well-defined representation of the Torelli group
\begin{equation}
\label{eq:representation-Torelli}
\mathfrak{T}(\Sigma) \longrightarrow \mathrm{Aut}_R \bigl( H_n^{BM} \bigl( \mathcal{C}_n(\Sigma) , \mathcal{C}_n(\Sigma,\partial^-(\Sigma)) ; V \bigr) \bigr)
\end{equation}
that lifts a projective action of $\mathfrak{T}(\Sigma)$ on this homology module.
\end{theorem}
\begin{proof}
Let us abbreviate $\mathcal{V}_n(V) = H_n^{BM} (\mathcal{C}_n(\Sigma) , \mathcal{C}_n(\Sigma,\partial^-(\Sigma)) ; V)$. The group homomorphism \eqref{eq:representation-Torelli-extension} must send the subgroup $\Z \subset \widetilde{\mathfrak{T}}(\Sigma)$ (the kernel of the central extension of $\mathfrak{T}(\Sigma)$) to the centre of $\mathrm{Aut}_R(\mathcal{V}_n(V))$, so it descends to
\begin{equation}
\mathfrak{T}(\Sigma) \longrightarrow \mathrm{PAut}_R(\mathcal{V}_n(V)),
\end{equation}
where the projective automorphism group $\mathrm{PAut}_R(A)$ of an $R$-module $A$ is the quotient of $\mathrm{Aut}_R(A)$ by its centre. Note that the centre of $\mathrm{Aut}_R(A)$ is equal to $\{ -\cdot\lambda \mid \lambda \in \mathcal{Z}(R^\times) \}$ when $A$ is a free $R$-module, but may be larger when $A$ is not free. This is a projective action of the Torelli group. To lift it to a linear action, we compose \eqref{eq:representation-Torelli-extension} with any section of the central extension $\widetilde{\mathfrak{T}}(\Sigma) \to \mathfrak{T}(\Sigma)$, which exists by Lemma \ref{lem:trivial-extension-of-Torelli}.
\end{proof}

To prove Lemma \ref{lem:trivial-extension-of-Torelli}, we will need a lemma describing the behaviour of the crossed homomorphism $\mathfrak{d}$ with respect to increasing genus. Consider the inclusion of surfaces $\Sigma_{g,1} \subseteq \Sigma_{h,1}$ given by boundary connected sum with $\Sigma_{h-g,1}$. This induces an inclusion of mapping class groups
\begin{equation}
\label{eq:inclusion-of-MCG}
\mathfrak{M}(\Sigma_{g,1}) \longhookrightarrow \mathfrak{M}(\Sigma_{h,1})
\end{equation}
by extending diffeomorphisms by the identity on $\Sigma_{h-g,1}$.

\begin{lemma}[{\cite[\S 5.2]{HeisenbergHomology}}]
\label{lem:Morita-crossed-hom-stabilisation}
The diagram
\begin{equation}
\label{eq:Morita-crossed-hom-stabilisation}
\begin{tikzcd}
\mathfrak{M}(\Sigma_{g,1}) \ar[rr,hook,"\eqref{eq:inclusion-of-MCG}"] \ar[d,"\mathfrak{d}",swap] && \mathfrak{M}(\Sigma_{h,1}) \ar[d,"\mathfrak{d}"] \\
H^1(\Sigma_{g,1};\Z) \ar[rr] && H^1(\Sigma_{h,1};\Z)
\end{tikzcd}
\end{equation}
commutes, where the bottom arrow is the map induced by the inclusion $\Sigma_{g,1} \hookrightarrow \Sigma_{h,1}$ on $H_1(-;\Z)$, conjugated by Poincar{\'e} duality.
\end{lemma}

\begin{proof}[Proof of Lemma \ref{lem:trivial-extension-of-Torelli}]
We begin by showing that it suffices to prove the statement for all sufficiently large $g$; we will then be able to assume $g\geq 3$ in the rest of the proof. For $g<h$, consider the inclusion of Torelli groups
\begin{equation}
\label{eq:stabilisation-Torelli}
\iota \colon \mathfrak{T}(\Sigma_{g,1}) \lhook\joinrel\longrightarrow \mathfrak{T}(\Sigma_{h,1}).
\end{equation}
We claim that the pullback of the central extension $\widetilde{\mathfrak{T}}(\Sigma_{h,1})$ along \eqref{eq:stabilisation-Torelli} is $\widetilde{\mathfrak{T}}(\Sigma_{g,1})$. To see this, recall from Remark \ref{rmk:cocycle-for-ext-of-Torelli} that the central extension $\widetilde{\mathfrak{T}}(\Sigma_{g,1})$ of $\mathfrak{T}(\Sigma_{g,1})$ is represented by the $2$-cocycle $(f,f') \mapsto \mathfrak{d}(f)^\sharp . \mathfrak{d}(f')^\sharp$. Similarly, the pullback of the central extension $\widetilde{\mathfrak{T}}(\Sigma_{h,1})$ of $\mathfrak{T}(\Sigma_{h,1})$ along the inclusion \eqref{eq:stabilisation-Torelli} is represented by the $2$-cocycle $(f,f') \mapsto \mathfrak{d}(\iota(f))^\sharp . \mathfrak{d}(\iota(f'))^\sharp$. Lemma \ref{lem:Morita-crossed-hom-stabilisation}, together with the fact that the map $H_1(\Sigma_{g,1};\Z) \to H_1(\Sigma_{h,1};\Z)$ preserves the intersection form, implies that these $2$-cocycles are equal. Thus triviality of $\widetilde{\mathfrak{T}}(\Sigma_{h,1})$ will imply triviality of $\widetilde{\mathfrak{T}}(\Sigma_{g,1})$ for any $g<h$. For the remainder of this proof, we assume that $g\geq 3$ and abbreviate $\Sigma_{g,1}$ to $\Sigma$, as usual.

By \cite[Lemma A.1(xiii)]{BensonCampagnoloRanickiRovi2020} and homological stability \cite[Theorem 1.2]{Wahl2013}, the canonical surjection $\mathfrak{M}(\Sigma) \twoheadrightarrow Sp(H)$ induces an isomorphism on $H^2(-;\Z)$ when $g\geq 3$. It follows that the inclusion $\mathfrak{T}(\Sigma) \hookrightarrow \mathfrak{M}(\Sigma)$ induces the trivial map on $H^2(-;\Z)$. This means that every $\Z$-central extension of $\mathfrak{M}(\Sigma)$ becomes trivial when restricted to $\mathfrak{T}(\Sigma)$. To prove the lemma, it will therefore suffice to show that $\widetilde{\mathfrak{T}}(\Sigma)$ is the restriction of a $\Z$-central extension defined on the whole mapping class group $\mathfrak{M}(\Sigma)$. Recall from Remark \ref{rmk:cocycle-for-ext-of-Torelli} that the central extension $\widetilde{\mathfrak{T}}(\Sigma)$ of $\mathfrak{T}(\Sigma)$ is represented by the $2$-cocycle $c'$ given by $c'(f,f') = \mathfrak{d}(f)^\sharp . \mathfrak{d}(f')^\sharp$. We therefore just have to show that the $2$-cocycle $c'$ extends to $\mathfrak{M}(\Sigma)$.

Now, there is a $2$-cocycle $c$ on $\mathfrak{M}(\Sigma)$, defined by Morita~\cite{Morita1989II}, given by the formula $c(f,f') = \mathfrak{d}(f^{-1})^\sharp . \mathfrak{d}(f')^\sharp$. By general properties of crossed homomorphisms, we have $\mathfrak{d}(f^{-1})^\sharp = -f_*^{-1}(\mathfrak{d}(f)^\sharp)$, and so we may rewrite this as
\begin{equation}
\label{eq:Morita-cocycle-rewritten}
c(f,f') = -f_*^{-1}(\mathfrak{d}(f)^\sharp) . \mathfrak{d}(f')^\sharp = -\mathfrak{d}(f)^\sharp . f_*(\mathfrak{d}(f')^\sharp),
\end{equation}
where for the second equality we have used the fact that the automorphism $f_*$ of $H$ preserves the intersection form. Restricted to the Torelli group, we have $f_* = \mathrm{id}$, so $c(f,f') = -\mathfrak{d}(f)^\sharp . \mathfrak{d}(f')^\sharp = -c'(f,f')$ for $f,f' \in \mathfrak{T}(\Sigma)$. Thus the $2$-cocycle $c'$ extends to the $2$-cocycle $-c$ on $\mathfrak{M}(\Sigma)$.
\end{proof}

\section{Representations of Earle-Morita subgroups}
\label{s:Morita}

We now identify another large subgroup of the mapping class group $\mathfrak{M}(\Sigma)$ on which we may construct linear representations without passing to a central extension. We will do this in the setting where we take coefficients in the \emph{Schrödinger representation} of $\Heis$ and the subgroup under consideration is the \emph{Earle-Morita subgroup} of $\mathfrak{M}(\Sigma)$.

Recall from \S\ref{subsec:automorphisms-of-Heis} that the \emph{Earle-Morita subgroup} $\mathrm{Mor}(\Sigma) \subseteq \mathfrak{M}(\Sigma)$ is the kernel of the crossed homomorphism $\mathfrak{d} \colon \mathfrak{M}(\Sigma) \to H^1(\Sigma;\Z)$ defined by Morita \cite{Morita1989}, and that this crossed homomorphism coincides with the one associated to the action $\mathfrak{M}(\Sigma) \to \mathrm{Aut}^+(\Heis) \cong Sp(H) \ltimes H^1(\Sigma;\Z)$ from Proposition \ref{f_Heisenberg}. We also recall, from Remark \ref{rmk:Earle-Morita-choices}, that this crossed homomorphism, and its kernel $\mathrm{Mor}(\Sigma)$, depend on the parametrisation of the surface $\Sigma$.

An important representation of the Heisenberg group is the \emph{Schrödinger representation}, which is parametrised by a non-zero real number $\hbar$ (called the Planck constant). It is given by the right action $\Pi_{\hbar}$ of $\Heis$ on the Hilbert space $W := L^{2}(\R^{g})$ determined by the following formula:

\begin{equation}
\label{eq:Schroedinger-formula}
\left[\Pi_{\hbar}\left(k,x=\sum_{i=1}^g p_i a_i + q_i b_i \right) \psi \right](s)=e^{i\hbar \frac{k-p\cdot q}{2}}e^{i\hbar p\cdot s}\psi (s-q) .
\end{equation}

In fact, this is an action of the \emph{continuous} Heisenberg group $\Heisr$, which is the central extension of $H_\R := H_1(\Sigma;\R)$ by $\R$ corresponding to the intersection form. There is a natural inclusion $\Heis \subset \Heisr$. The Schrödinger representation is a unitary action on $W = L^{2}(\R^{g})$, so it may be written as
\begin{equation}
\label{eq:Schroedinger-rep}
\Pi_{\hbar} \colon \Heisr \longrightarrow U(W) .
\end{equation}

We recall also that the group $\mathrm{Aut}^+(\Heisr)$ of automorphisms acting trivially on the centre of $\Heisr$ decomposes as $\mathrm{Aut}^+(\Heisr) \cong Sp(H_\R) \ltimes H^1(\Sigma;\R)$, similarly to the decomposition of $\mathrm{Aut}^+(\Heis)$ described in \S\ref{subsec:automorphisms-of-Heis}. From these decompositions we see that there is a natural inclusion $\mathrm{Aut}^+(\Heis) \subset \mathrm{Aut}^+(\Heisr)$.

\subsection{Untwisting on a central extension of the mapping class group.}
\label{subsec:untwisting-Schroedinger}

We first recall from \cite[\S 5]{HeisenbergHomology} how to untwist the twisted representation \eqref{eq:twisted-representation} on the \emph{stably universal central extension} of $\mathfrak{M}(\Sigma)$ when the $\Heis$-representation $V$ is the Schrödinger representation. In \S\ref{subsec:metaplectic}--\S\ref{subsec:untwisting-on-Earle-Morita} we then explain how to untwist on the Earle-Morita subgroup \emph{without} passing to a central extension.

As recalled in \cite[\S 5.1]{HeisenbergHomology}, an immediate corollary of the \emph{Stone-von Neumann theorem} (see for example \cite[p.~19]{LionVergne}) is the following.

\begin{corollary}[of the Stone-von Neumann theorem]
\label{coro:Stone-von-Neumann}
Fix a positive real number $\hbar$ and let $\rho \colon \Heisr \to U(W)$ be an irreducible unitary representation whose restriction to the centre $\R \subset \Heisr$ is given by $\rho(t,0) = e^{\hbar it/2}.\mathrm{id}_W$. Then there is an element $u \in U(W)$, unique up to rescaling by an element of $S^1$, such that $\rho = u . \Pi_{\hbar} . u^{-1}$.
\end{corollary}

In particular, we may apply this result to the representation $\rho := \Pi_{\hbar} \circ \varphi$, for any $\varphi \in \mathrm{Aut}(\Heisr)$. Sending $\varphi$ to the element $u \in U(W)/S^1 = PU(W)$ provided by Corollary \ref{coro:Stone-von-Neumann} defines a homomorphism
\begin{equation}
\label{eq:Segal-Shale-Weil-representation}
T \colon \mathrm{Aut}(\Heisr) \longrightarrow PU(W),
\end{equation}
which is the \emph{Segal-Shale-Weil projective representation}. Restricting to $\mathrm{Aut}^+(\Heis) \subset \mathrm{Aut}^+(\Heisr) \subset \mathrm{Aut}(\Heisr)$, we may compose it with the action $\Psi \colon \mathfrak{M}(\Sigma) \to \mathrm{Aut}^+(\Heis)$ from Proposition \ref{f_Heisenberg} to obtain a projective representation
\begin{equation}
\label{eq:Segal-Shale-Weil-representation-MCG}
\mathfrak{M}(\Sigma) \longrightarrow PU(W)
\end{equation}
of the mapping class group. This is the key ingredient for untwisting the twisted representations of the mapping class group constructed in \S\ref{subsec:twisted-representations}.

\begin{definition}
Let $\overline{\mathfrak{M}}(\Sigma)$ denote the central extension of $\mathfrak{M}(\Sigma)$ by $S^1$ pulled back from the central extension $U(W)$ of $PU(W)$ along \eqref{eq:Segal-Shale-Weil-representation-MCG}. By construction, the projective representation \eqref{eq:Segal-Shale-Weil-representation-MCG} lifts to a linear representation
\begin{equation}
\label{eq:Segal-Shale-Weil-representation-MCG-lifted}
\overline{\mathfrak{M}}(\Sigma) \longrightarrow U(W)
\end{equation}
on this central extension.
\end{definition}

\begin{definition}
For $g\geq 4$, the mapping class group $\mathfrak{M}(\Sigma)$ is perfect and we have $H_2(\mathfrak{M}(\Sigma);\Z) \cong \Z$, so it has a universal central extension with kernel $\Z$. Let us denote this extension by $\widetilde{\mathfrak{M}}(\Sigma)$. For $h\geq g\geq 4$, the pullback of $\widetilde{\mathfrak{M}}(\Sigma_{h,1})$ along the inclusion \eqref{eq:inclusion-of-MCG} is $\widetilde{\mathfrak{M}}(\Sigma_{g,1})$. Thus we may define, for $g\geq 1$, the \emph{stably universal central extension} $\widetilde{\mathfrak{M}}(\Sigma_{g,1})$ of $\mathfrak{M}(\Sigma_{g,1})$ to be the pullback of the universal central extension $\widetilde{\mathfrak{M}}(\Sigma_{h,1})$ of $\mathfrak{M}(\Sigma_{h,1})$ along the inclusion \eqref{eq:inclusion-of-MCG}, for any $h\geq 4$.
\end{definition}

We note that there is a canonical morphism of central extensions
\begin{equation}
\label{eq:morphism-of-extensions}
\widetilde{\mathfrak{M}}(\Sigma) \longrightarrow \overline{\mathfrak{M}}(\Sigma).
\end{equation}
When $g\geq 4$ this morphism exists and is unique by universality of $\widetilde{\mathfrak{M}}(\Sigma)$. For $g\leq 3$ it may be pulled back from the $g\geq 4$ case via the inclusion \eqref{eq:inclusion-of-MCG}, as explained in \cite[\S 5.3]{HeisenbergHomology}.

Using these ingredients, we showed in \cite[\S 5.3]{HeisenbergHomology} how to untwist the twisted representation \eqref{eq:twisted-representation} of $\mathfrak{M}(\Sigma)$ when taking coefficients in the Schrödinger representation $W$ of $\Heis$, after passing to the stably universal central extension $\widetilde{\mathfrak{M}}(\Sigma)$. Let us recall briefly how this works, following the philosophy of Remark \ref{rmk:untwisting}. In the notation of \S\ref{subsec:twisted-representations}, the construction of the twisted representation \eqref{eq:twisted-representation} provides isomorphisms
\begin{equation}
\label{eq:isomorphism-from-twisted-representation}
\mathcal{V}_n \bigl( {}_{f_\Heis} \! V \bigr) \longrightarrow \mathcal{V}_n(V)
\end{equation}
for each $f \in \mathfrak{M}(\Sigma)$. (This is simply \eqref{eq:action-of-f} with $\tau = \mathrm{id}_\Heis$.) For each $\bar{f} \in \overline{\mathfrak{M}}(\Sigma)$ lifting $f \in \mathfrak{M}(\Sigma)$, its image under \eqref{eq:Segal-Shale-Weil-representation-MCG-lifted} is a unitary automorphism of $W$ that intertwines the left Schrödinger action of $\Heis$, as long as we twist the action on the codomain by $f_\Heis$. In other words, it is a (unitary) isomorphism of left $\Heis$-representations of the form $W \cong {}_{f_\Heis} \! W$. This is a direct consequence of the defining property of the Segal-Shale-Weil projective representation from Corollary \ref{coro:Stone-von-Neumann}. This isomorphism of coefficients induces an isomorphism $\mathcal{V}_n(W) \cong \mathcal{V}_n({}_{f_\Heis} \! W)$, which we may compose with \eqref{eq:isomorphism-from-twisted-representation} (for $V=W$) to obtain:

\begin{theorem}[{\cite[\S 5]{HeisenbergHomology}}]
\label{thm:Schroedinger}
For any $n\geq 2$, there is a representation
\begin{equation}
\label{eq:MCG-representation-from-Schroedinger}
\overline{\mathfrak{M}}(\Sigma) \longrightarrow GL(\mathcal{V}_n(W))
\end{equation}
induced by the natural action of the mapping class group on \eqref{eq:Borel-Moore-homology-group} with coefficients in the Schrödinger representation $V=W$. Via the morphism \eqref{eq:morphism-of-extensions}, we may view this as a representation of the stably universal central extension of $\mathfrak{M}(\Sigma)$.
\end{theorem}

The purpose of this section is to show that, when we restrict to the Earle-Morita subgroup $\mathrm{Mor}(\Sigma) \subseteq \mathfrak{M}(\Sigma)$ and take coefficients in the Schrödinger representation $V=W$, we may obtain a representation of $\mathrm{Mor}(\Sigma)$ itself, \emph{without} passing to any central extension.

We shall do this as follows. We first recall the \emph{metaplectic extension} $\widehat{\mathfrak{M}}(\Sigma)$ of the mapping class group, which is an extension by $\Z/2$. Denoting by $\widehat{\mathrm{Mor}}(\Sigma)$ and $\overline{\mathrm{Mor}}(\Sigma)$ the restrictions of $\widehat{\mathfrak{M}}(\Sigma)$ and $\overline{\mathfrak{M}}(\Sigma)$ to the Earle-Morita subgroup, we show that $\overline{\mathrm{Mor}}(\Sigma)$ contains $\widehat{\mathrm{Mor}}(\Sigma)$ and that $\widehat{\mathrm{Mor}}(\Sigma)$ is a trivial extension. It will then follow that we may restrict \eqref{eq:MCG-representation-from-Schroedinger} to $\widehat{\mathrm{Mor}}(\Sigma) \subset \overline{\mathrm{Mor}}(\Sigma) \subset \overline{\mathfrak{M}}(\Sigma)$ and pre-compose with a section of the trivial extension $\widehat{\mathrm{Mor}}(\Sigma)$ to obtain a representation of the Earle-Morita subgroup $\mathrm{Mor}(\Sigma)$.

\subsection{Metaplectic extensions.}
\label{subsec:metaplectic}

We first consider two extensions of the symplectic group $Sp(H_\R)$.

\begin{definition}
Recall that the fundamental group of $Sp(H_\R) \cong Sp_{2g}(\R)$ is infinite cyclic. It therefore has a unique connected double covering group, which is called the \emph{metaplectic group}, denoted by $Mp(H_\R)$.
\end{definition}

\begin{definition}
Consider the restriction of the projective representation \eqref{eq:Segal-Shale-Weil-representation} to the subgroup $Sp(H_\R) \subset Sp(H_\R) \ltimes H^1(\Sigma;\Z) \cong \mathrm{Aut}^+(\Heisr) \subset \mathrm{Aut}(\Heisr)$, which is a projective representation
\begin{equation}
\label{eq:Segal-Shale-Weil-representation-restricted}
Sp(H_\R) \longrightarrow PU(W)
\end{equation}
of the symplectic group. It is in fact this restriction that is more usually referred to by the name \emph{Segal-Shale-Weil projective representation}. Denote by $\overline{Sp}(H_\R)$ the pullback of the central extension $U(W)$ of $PU(W)$. This is a central extension of $Sp(H_\R)$ by $S^1$.
\end{definition}

The main technical result of this subsection is the following:

\begin{proposition}
\label{prop:inclusion-of-extensions}
There is an inclusion $Mp(H_\R) \subset \overline{Sp}(H_\R)$ of central extensions, restricting to the inclusion $\Z/2 \cong \{\pm 1\} \subset S^1$ on fibres.
\end{proposition}

Before proving this, we record its implications under pulling back to the mapping class group. We first define the relevant extensions of the mapping class group and its subgroups.

\begin{definition}
Denote by $\mathfrak{s} \colon \mathfrak{M}(\Sigma) \to Sp(H)$ the standard symplectic action of the mapping class group on $H = H_1(\Sigma;\Z)$. The \emph{metaplectic extension} $\widehat{\mathfrak{M}}(\Sigma)$ of $\mathfrak{M}(\Sigma)$ is defined to be its central extension by $\Z/2$ given by pulling back the $\Z/2$-central extension $Mp(H_\R) \to Sp(H_\R)$ along $\mathfrak{s}$ and the inclusion $Sp(H) \subset Sp(H_\R)$.
\end{definition}

\begin{definition}
For a subgroup $G \subseteq \mathfrak{M}(\Sigma)$, we denote the restrictions of the central extensions $\overline{\mathfrak{M}}(\Sigma)$ and $\widehat{\mathfrak{M}}(\Sigma)$ to $G$ by $\overline{G}$ and $\widehat{G}$ respectively.
\end{definition}

\begin{corollary}
\label{coro:inclusions-of-extensions}
There is an inclusion $\widehat{\mathrm{Mor}}(\Sigma) \subset \overline{\mathrm{Mor}}(\Sigma)$ of central extensions, restricting to the inclusion $\Z/2 \cong \{\pm 1\} \subset S^1$ on fibres.
\end{corollary}
\begin{proof}
By definition, the central extension $\overline{\mathfrak{M}}(\Sigma)$ of $\mathfrak{M}(\Sigma)$ is pulled back from the central extension $U(W)$ of $PU(W)$ along the top row of the following diagram.
\begin{equation*}
\begin{tikzpicture}
[x=1mm,y=1.3mm]
\node (t2) at (30,20) {$\mathrm{Aut}^+(\Heis)$};
\node (t3) at (70,20) {$\mathrm{Aut}^+(\Heisr)$};
\node (t4) at (100,20) {$PU(W)$};
\node (m1) at (-5,10) {$\mathfrak{M}(\Sigma)$};
\node (m2) at (30,10) {$Sp(H) \ltimes H^1(\Sigma;\Z)$};
\node (m3) at (70,10) {$Sp(H_\R) \ltimes H^1(\Sigma;\R)$};
\node (b1) at (-5,0) {$\mathrm{Mor}(\Sigma)$};
\node (b2) at (30,0) {$Sp(H)$};
\node (b3) at (70,0) {$Sp(H_\R)$};
\draw[->] (m1) to node[above left,font=\small]{$\Psi$} (t2.west);
\incl{(t2)}{(t3)}
\draw[->] (t3) to node[above,font=\small]{$T$} (t4);
\draw[->] (m1) to node[below,font=\small]{$(\mathfrak{s},\mathfrak{d})$} (m2);
\incl{(m2)}{(m3)}
\incl{(b1)}{(m1)}
\draw[->] (b1) to node[above,font=\small]{$\mathfrak{s}$} (b2);
\incl{(b2)}{(m2)}
\incl{(b3)}{(m3)}
\incl{(b2)}{(b3)}
\node at (30,15) {\rotatebox{90}{$\cong$}};
\node at (67.5,15) {\rotatebox{90}{$\cong$}};
\end{tikzpicture}
\end{equation*}
Since this diagram commutes (for the bottom-left square this is because $\mathfrak{d} \equiv 0$ on $\mathrm{Mor}(\Sigma)$), it follows that its restriction $\overline{\mathrm{Mor}}(\Sigma)$ to the Earle-Morita subgroup is the pullback of the central extension $\overline{Sp}(H_\R)$ of $Sp(H_\R)$ along the bottom row of the diagram. On the other hand, the metaplectic extension $\widehat{\mathrm{Mor}}(\Sigma)$ is by definition the pullback of $Mp(H_\R)$ along the bottom row of the diagram. Thus the inclusion of central extensions $Mp(H_\R) \subset \overline{Sp}(H_\R)$ of $Sp(H_\R)$ from Proposition \ref{prop:inclusion-of-extensions} pulls back to the desired inclusion $\widehat{\mathrm{Mor}}(\Sigma) \subset \overline{\mathrm{Mor}}(\Sigma)$ of central extensions of $\mathrm{Mor}(\Sigma)$.
\end{proof}

\begin{remark}
The argument above does \emph{not} show that the metaplectic extension includes into the $S^1$-extension pulled back via the Segal-Shale-Weil representation on the \emph{full} mapping class group. This is because the inclusion of central extensions essentially arises at the level of the symplectic group (Proposition \ref{prop:inclusion-of-extensions}), so we have to restrict to the kernel of $\mathfrak{d}$ to ensure that the $S^1$-extension pulls back via the symplectic group.
\end{remark}

\begin{proof}[Proof of Proposition \ref{prop:inclusion-of-extensions}]
Let us first slightly rewrite the statement in notation that makes the dependence on $g$ explicit: our goal is to prove that, over the group $Sp_{2g}(\R)$, there is an embedding of central extensions $Mp_{2g}(\R) \hookrightarrow \overline{Sp}_{2g}(\R)$ (which must necessarily restrict to the inclusion $\Z/2 \cong \{\pm 1\} \subset S^1$ on fibres).

We first show that it suffices to prove this statement for all $g$ sufficiently large; we will then be able to assume for the rest of the proof that $g\geq 4$, which is the stable range for (co)homology of degree at most $2$ for $Sp_{2g}(\R)$ and $\mathfrak{M}(\Sigma_{g,1})$. For any $g<h$ there is an inclusion map $Sp_{2g}(\R) \hookrightarrow Sp_{2h}(\R)$ given by extending symplectic automorphisms of $\R^{2g}$ by the identity on $\R^{2h-2g}$. We claim that the pullbacks of $Mp_{2h}(\R)$ and of $\overline{Sp}_{2h}(\R)$ under this inclusion are $Mp_{2g}(\R)$ and $\overline{Sp}_{2g}(\R)$ respectively. For the metaplectic central extensions this follows from the fact that the induced map $\pi_1(Sp_{2g}(\R)) \cong \Z \to \Z \cong \pi_1(Sp_{2h}(\R))$ is an isomorphism and the metaplectic double covering corresponds to the unique index-$2$ subgroup of $\pi_1$. For $\overline{Sp}$, note that the Segal-Shale-Weil projective representations in genus $g$ and $h$ fit into a commutative square as follows:
\begin{equation}
\label{eq:Shale-Weil-representation-stabilisation}
\begin{tikzcd}
Sp_{2g}(\R) \ar[rrr,"\eqref{eq:Segal-Shale-Weil-representation-restricted}"] \ar[d,hook] &&& PU(L^2(\R^g)) \ar[d] & U(L^2(\R^g)) \ar[l,two heads] \ar[d] \\
Sp_{2h}(\R) \ar[rrr,"\eqref{eq:Segal-Shale-Weil-representation-restricted}"] &&& PU(L^2(\R^h)) & U(L^2(\R^h)) \ar[l,two heads]
\end{tikzcd}
\end{equation}
The right-hand side square of this diagram arises as follows. We consider $L^2(\R^g)$ as the closed subspace of $L^2(\R^h)$ consisting of those $L^2$-functions that factor through $\R^h = \R^g \times \R^{h-g} \twoheadrightarrow \R^g$. Any closed subspace of a Hilbert space has an orthogonal complement, so we may extend unitary automorphisms of $L^2(\R^g)$ by the identity on this complement to obtain a homomorphism $U(L^2(\R^g)) \to U(L^2(\R^h))$, which descends to the projective unitary groups, forming a pullback square. By definition, $\overline{Sp}_{2g}(\R)$ is the pullback along \eqref{eq:Segal-Shale-Weil-representation-restricted} of the extension $U(L^2(\R^g))$ of $PU(L^2(\R^g))$. Commutativity of \eqref{eq:Shale-Weil-representation-stabilisation} then implies that the pullback of $\overline{Sp}_{2h}(\R)$ along the inclusion is $\overline{Sp}_{2g}(\R)$. Thus the existence of an embedding $Mp_{2h}(\R) \hookrightarrow \overline{Sp}_{2h}(\R)$ will imply the existence of an embedding $Mp_{2g}(\R) \hookrightarrow \overline{Sp}_{2g}(\R)$ for $g<h$. We henceforth assume that $g\geq 4$ in this proof (this is only needed in the last paragraph).

First, we recall from \cite[\S 1.7]{LionVergne} that a particular choice of cocycle
\[
\omega_{Sp} \colon Sp_{2g}(\R) \times Sp_{2g}(\R) \longrightarrow S^1 ,
\]
representing the central extension $\overline{Sp}_{2g}(\R)$, takes values in the finite cyclic subgroup $\Z/8 \subseteq S^1$, so there is an embedding of central extensions $Sp_{2g}(\R)^{(8)} \hookrightarrow \overline{Sp}_{2g}(\R)$, for a certain $\Z/8$-central extension $Sp_{2g}(\R)^{(8)}$ of $Sp_{2g}(\R)$. Moreover, this central extension is classified by the element $-[\tau].8\Z \in H^2(Sp_{2g}(\R);\Z/8)$, the reduction modulo $8$ of the element $-[\tau] \in H^2(Sp_{2g}(\R);\Z)$ represented by the negative of the \emph{Maslov cocycle} $\tau$ (see formula 1.7.7 on page 70 of \cite{LionVergne}).

Second, we also recall from \cite[\S 1.7]{LionVergne} that there is a function
\[
s \colon Sp_{2g}(\R) \longrightarrow \Z/4 \subseteq S^1
\]
such that $\omega_{Sp}(g,h)^2 = s(g)^{-1}s(h)^{-1}s(gh)$ (formula 1.7.8 on page 70 of \cite{LionVergne}). It follows that the subset of $Sp_{2g}(\R)^{(8)}$ of those pairs $(t,g)$ for which $t^2 = s(g)$ is a subgroup. The projection onto $Sp_{2g}(\R)$ restricted to this subgroup is a double covering, and so this subgroup must either be the trivial covering $Sp_{2g}(\R) \times \Z/2$ or the metaplectic covering $Mp_{2g}(\R)$.

To finish the proof, we just have to show that it cannot be the trivial covering. Suppose for a contradiction that it is. Then $Sp_{2g}(\R)^{(8)}$ admits a section, so it is a trivial extension and we must have $[\tau].8\Z = 0 \in H^2(Sp_{2g}(\R);\Z/8)$. However, the pullback of $[\tau]$ along the projection $\mathfrak{M}(\Sigma) \to Sp_{2g}(\R)$, also denoted by $[\tau]$, is precisely $4$ times a generator of $H^2(\mathfrak{M}(\Sigma);\Z) \cong \Z$ (here we are using the assumption that $g\geq 4$). Thus $[\tau].8\Z \in H^2(\mathfrak{M}(\Sigma);\Z/8) \cong \Z/8$ is non-zero. Hence we must have $[\tau].8\Z \neq 0$ already in $H^2(Sp_{2g}(\R);\Z/8)$. This completes the proof.
\end{proof}

\subsection{Triviality of an extension.}

The last ingredient that we will need is the following.

\begin{proposition}
\label{prop:metaplectic-trivial-extension}
The $\Z/2$-central extension $\widehat{\mathrm{Mor}}(\Sigma)$ of $\mathrm{Mor}(\Sigma)$ is trivial.
\end{proposition}
\begin{proof}
We first note that it suffices to prove this statement for all sufficiently large $g$. This is because the inclusion of mapping class groups $\mathfrak{M}(\Sigma_{g,1}) \hookrightarrow \mathfrak{M}(\Sigma_{h,1})$ restricts to an inclusion of Earle-Morita subgroups $\mathrm{Mor}(\Sigma_{g,1}) \hookrightarrow \mathrm{Mor}(\Sigma_{h,1})$ (as an immediate consequence of Lemma \ref{lem:Morita-crossed-hom-stabilisation}), and the pullback of $\widehat{\mathrm{Mor}}(\Sigma_{h,1})$ along this inclusion is $\widehat{\mathrm{Mor}}(\Sigma_{g,1})$, for any $g<h$. This last statement follows from the fact that the pullback of $Mp_{2h}(\R)$ along $Sp_{2g}(\R) \hookrightarrow Sp_{2h}(\R)$ is $Mp_{2g}(\R)$, which was explained during the proof of Proposition \ref{prop:inclusion-of-extensions}. We now assume that $g\geq 4$ for the rest of the proof.

Recall from the proof of Proposition \ref{prop:inclusion-of-extensions} that there is an embedding of central extensions $Mp_{2g}(\R) \hookrightarrow Sp_{2g}(\R)^{(8)}$, where $Sp_{2g}(\R)^{(8)}$ is a certain central extension of $Sp_{2g}(\R)$ by $\Z/8$. Pulling back along the symplectic action $\mathfrak{M}(\Sigma) \to Sp_{2g}(\R)$, we obtain an embedding of central extensions $\widehat{\mathfrak{M}}(\Sigma) \hookrightarrow \mathfrak{M}(\Sigma)^{(8)}$, where $\mathfrak{M}(\Sigma)^{(8)}$ is classified by $-[\tau].8\Z \in H^2(\mathfrak{M}(\Sigma);\Z/8) \cong \Z/8$. Now, $H^2(\mathfrak{M}(\Sigma);\Z)$ is infinite cyclic, generated by the first Chern class $c_1$, and we have $[\tau] = 4c_1$. There is also a cocycle $c \colon \mathfrak{M}(\Sigma) \times \mathfrak{M}(\Sigma) \to \Z$ defined by Morita~\cite{Morita1989II} given by the formula $c(f,f') = \mathfrak{d}(f^{-1})^\sharp . \mathfrak{d}(f')^\sharp$ (see also the proof of Lemma \ref{lem:trivial-extension-of-Torelli}) and we have $[c] = 12c_1$ in $H^2(\mathfrak{M}(\Sigma);\Z)$. Thus, in particular, we have $3[\tau] = [c]$. Since $\mathrm{Mor}(\Sigma) = \mathrm{ker}(\mathfrak{d})$, Morita's cocycle $c$ vanishes on $\mathrm{Mor}(\Sigma)$, and so after restricting to the Earle-Morita subgroup we have $3[\tau] = [c] = 0 \in H^2(\mathrm{Mor}(\Sigma);\Z)$. Reducing modulo $8$ we therefore have $3[\tau].8\Z = 0 \in H^2(\mathrm{Mor}(\Sigma);\Z/8)$. But this cohomology group is a $\Z/8$-module, and $3$ is invertible modulo $8$, so we may divide by $3$ and deduce that $[\tau].8\Z = 0$ in $H^2(\mathrm{Mor}(\Sigma);\Z/8)$. Hence the restriction $\mathrm{Mor}(\Sigma)^{(8)}$ of $\mathfrak{M}(\Sigma)^{(8)}$ to the Earle-Morita subgroup $\mathrm{Mor}(\Sigma)$ is a trivial extension. It therefore follows from the embedding $\widehat{\mathrm{Mor}}(\Sigma) \hookrightarrow \mathrm{Mor}(\Sigma)^{(8)}$ that $\widehat{\mathrm{Mor}}(\Sigma)$ is also a trivial extension.
\end{proof}

\begin{remark}
In summary, we have considered, in this subsection and the previous one, three nested central extensions $Mp(H_\R) \subset Sp(H_\R)^{(8)} \subset \overline{Sp}(H_\R)$ of the symplectic group $Sp(H_\R)$ with fibres $\Z/2 \subset \Z/8 \subset S^1$. Clearly they are either all trivial or all non-trivial. We have seen that their pullbacks along the symplectic action $\mathfrak{M}(\Sigma) \to Sp(H_\R)$ are non-trivial (and hence they must also be non-trivial to begin with), but their further pullbacks (restrictions) to the Earle-Morita subgroup $\mathrm{Mor}(\Sigma) \subset \mathfrak{M}(\Sigma)$ are trivial.
\end{remark}

\subsection{Untwisted representations of Earle-Morita subgroups.}
\label{subsec:untwisting-on-Earle-Morita}

We may now conclude with the main result of this section:

\begin{theorem}
\label{thm:Morita}
For any $n\geq 2$, there is a representation
\begin{equation}
\label{eq:Morita-representation-from-Schroedinger}
\mathrm{Mor}(\Sigma) \longrightarrow GL(\mathcal{V}_n(W))
\end{equation}
induced by the natural action of the mapping class group on the twisted Borel-Moore homology group \eqref{eq:Borel-Moore-homology-group} with coefficients in the Schrödinger representation $V=W$.
\end{theorem}
\begin{proof}
By Theorem \ref{thm:Schroedinger}, such a representation is defined on the central extension $\overline{\mathfrak{M}}(\Sigma)$ of the mapping class group by $S^1$. By Corollary \ref{coro:inclusions-of-extensions}, the restriction of this extension to $\mathrm{Mor}(\Sigma) \subset \mathfrak{M}(\Sigma)$ contains the metaplectic extension $\widehat{\mathrm{Mor}}(\Sigma)$, so we may further restrict to this subgroup. By Proposition \ref{prop:metaplectic-trivial-extension}, the central extension $\widehat{\mathrm{Mor}}(\Sigma)$ of $\mathrm{Mor}(\Sigma)$ is trivial, i.e., it admits a section. Hence, composing with any such section, we obtain the desired representation \eqref{eq:Morita-representation-from-Schroedinger}.
\end{proof}

\bibliographystyle{amsalpha}
\bibliography{biblio}

\end{document}